\documentclass[11pt]{amsart}   	
\usepackage{geometry}             
\usepackage[dvipsnames]{xcolor}   		             		
\usepackage{graphicx}				
\usepackage{amssymb,mathrsfs}
\usepackage{amsthm,amsmath,stmaryrd}
\usepackage{tikz}
\usepackage{tikz-cd}
\usepackage{accents,upgreek,enumerate}
\usepackage{booktabs}
\usepackage[headings]{fullpage}
\usepackage{bm}
\usepackage{mathtools}
\usepackage[all]{xy}
\usepackage{caption}
\usepackage{standalone}
\usepackage[euler-digits]{eulervm}
\usepackage[normalem]{ulem}
\usepackage{dirtytalk}
\usepackage{float}

\usepackage{setspace}
\tikzset{
  commutative diagrams/.cd, 
  arrow style=tikz, 
  diagrams={>=stealth}
}

\tikzset{every picture/.style={line width=0.75pt}} 

\usetikzlibrary{decorations.markings}
\tikzset{arrow/.style={
        decoration={markings,
            mark= at position 0.5 with {\arrow[scale=0.8]{>}} ,
        },
        postaction={decorate}
    }
}

\newtheorem{innercustomthm}{Theorem}
\newenvironment{customthm}[1]
  {\renewcommand\theinnercustomthm{#1}\innercustomthm}
  {\endinnercustomthm}

  \makeatletter
\def\@tocline#1#2#3#4#5#6#7{\relax
  \ifnum #1>\c@tocdepth 
  \else
    \par \addpenalty\@secpenalty\addvspace{#2}%
    \begingroup \hyphenpenalty\@M
    \@ifempty{#4}{%
      \@tempdima\csname r@tocindent\number#1\endcsname\relax
    }{%
      \@tempdima#4\relax
    }%
    \parindent\z@ \leftskip#3\relax \advance\leftskip\@tempdima\relax
    \rightskip\@pnumwidth plus4em \parfillskip-\@pnumwidth
    #5\leavevmode\hskip-\@tempdima
      \ifcase #1
       \or\or \hskip 1em \or \hskip 2em \else \hskip 3em \fi%
      #6\nobreak\relax
    \dotfill\hbox to\@pnumwidth{\@tocpagenum{#7}}\par
    \nobreak
    \endgroup
  \fi}
\makeatother

\usetikzlibrary{calc}
\usetikzlibrary{fadings}
\usetikzlibrary{decorations.pathmorphing}
\usetikzlibrary{decorations.pathreplacing}

\newcounter{marginnote}
\DeclareMathAlphabet{\mathpzc}{OT1}{pzc}{m}{it}

\usepackage[backref=page]{hyperref}
\hypersetup{
  colorlinks   = true,          
  urlcolor     = violet,          
  linkcolor    = violet,          
  citecolor   = blue             
}

\newtheorem{theorem}{Theorem}[subsection]
\newtheorem{corollary}[theorem]{Corollary}
\newtheorem{lemma}[theorem]{Lemma}
\newtheorem*{lemma*}{Lemma}
\newtheorem{proposition}[theorem]{Proposition}
\newtheorem*{proposition*}{Proposition}

\newtheorem{quasi-theorem}[theorem]{Quasi-Theorem}

\theoremstyle{definition}
\newtheorem{definition}[theorem]{Definition}

\newtheorem{remark}[theorem]{Remark}

\newtheorem{construction}[theorem]{Construction}

\newtheorem{blank remark}[theorem]{}

\newtheorem{not1}[theorem]{Notation}

\newcommand{\PP}{\mathbb{P}}         
\newcommand{\QQ} {{\mathbb Q}}		
\newcommand{\RR} {{\mathbb R}}		
\newcommand{\ZZ} {{\mathbb Z}}		

\def\setminus{\smallsetminus}

\newcommand{\Hom}{\operatorname{Hom}}

\DeclareMathOperator{\spec}{Spec}

\DeclareMathOperator{\Star}{Star}

\newcommand{\cal}{\mathcal}

\def\cM{{\cal M}}

\newcommand{\Mbar}{\overline{\cM}\vphantom{\cM}}

\def\trop{\operatorname{trop}}

\newcommand{\Spec}{\operatorname{Spec}}

\makeatletter
\def\blfootnote{\xdef\@thefnmark{}\@footnotetext}
\makeatother

\title{Combinatorics of Hurwitz degenerations and tropical realizability}

\date{}

\author{Mia Lam, Chi Kin Ng, {\it \&} Dhruv Ranganathan}

\address{{\bf Mia Lam}\newline School of Mathematics\newline University of Edinburgh, Edinburgh, UK}
\email{\href{mailto:Mia.Lam@ed.ac.uk}{Mia.Lam@ed.ac.uk}}

\address{{\bf Dhruv Ranganathan} \newline Department of Pure Mathematics {\it \&} Mathematical Statistics\newline
University of Cambridge, Cambridge, UK}
\email{\href{mailto:dr508@cam.ac.uk}{dr508@cam.ac.uk}}

\address{{\bf Chi Kin Ng} \newline Mathematics Department \newline
Northwestern University, Evanston, IL, USA}
\email{\href{mailto:ching2029@u.northwestern.edu}{ching2029@u.northwestern.edu}}

\begin{document}

\begin{abstract}
    We investigate the realizability of balanced functions on tropical curves, establishing new sufficient criteria for superabundant functions on genus two curves, analogous to the well-spacedness condition in genus one. We find that realizability is sensitive to the precise locations of conjugate and Weierstrass points on the tropical curve. The key input is a combinatorial comparison of semistable limit theorems for maps of curves. Amini--Baker--Brugall\'e--Rabinoff previously showed that realizability of functions is equivalent to \emph{modifiability} to a tropical admissible cover. While the resulting criteria are typically inexplicit, we develop combinatorial techniques to derive explicit, verifiable conditions. We further introduce a dimensional reduction technique to deduce statements about maps to $\mathbb{R}^r$ from corresponding statements about maps to $\mathbb{R}$. By proving directly that modifiability and well-spacedness are equivalent in genus one, we obtain a new proof that well-spaced maps are realizable. Along the way, we explain how the modifiability criterion can be interpreted as a comparison result for properness statements in moduli spaces of relative maps and admissible covers.
\end{abstract}

\maketitle

\setcounter{tocdepth}{1}
\tableofcontents

\section*{Introduction}
In this paper, we use the combinatorics of degenerations of Hurwitz covers to study the tropical realizability problem for maps from tropical curves. We establish new realizability criteria for curves of genus two and provide new proofs of several known realizability theorems using these techniques.  

\subsection{The problem} 
The tropical inverse problem is a fundamental question in tropical geometry: when does a synthetically defined tropical object arise from algebraic geometry? A key instance concerns the realizability of \emph{parameterized tropical curves}, see~\cite{BPR16,CFPU,Mi03,Ni09,Ni15,NS06}. Fix an abstract tropical curve $\Gamma$ of genus $g$ and a balanced map
\[
F\colon \Gamma \to \mathbb{R}^r.
\]
We say that $F$ is \emph{realizable} if there exists a smooth curve $\mathcal{C}$ of genus $g$, defined over a non-archimedean field, and a rational map
\[
\varphi\colon \mathcal{C} \dashrightarrow \mathbb{G}_m^r,
\]
such that the tropicalization of $\varphi$ equals $F$. The relevant basic notions of tropical geometry are recalled in Section~\ref{sec: modifications}. In this paper, we focus on the case where all vertex genera are equal to $0$. 

\subsection{State of the art} 
To contextualize our results, we explain how the complexity of the realizability problem grows with the genus $g$ and the target dimension $r$. When $g = 0$, for any value of $r$, all balanced tropical maps are realizable~\cite{NS06}. More generally, an explicit combinatorial condition called \emph{non-superabundance} is sufficient for realizability for all values of $g$ and $r$~\cite{CFPU}. The (non-)superabundance condition depends only on the combinatorics of the map: the underlying graph and the directions of edges under the map.  

However, \emph{superabundant maps} are expected to be generic objects, and understanding their realizability lies at the heart of the tropical inverse problem. Superabundant phenomena appear for each value of $(g,r)$ starting at $(1,1)$. For a superabundant configuration, additional conditions on the edge lengths of the tropical curve are typically required to guarantee realizability. The basic goal is to identify these additional \emph{realizability conditions}.  

Once a superabundant configuration appears, it ``propagates'' to all higher genera and target dimensions: smaller genus graphs appear as subgraphs of larger genus graphs, and maps to lower-dimensional targets appear as projections of maps to higher-dimensional ones. For each value of $(g,r)$, new phenomena arise that are not inherited from lower genus or lower dimension, see~\cite{Koy23}. Realizability conditions can also be propagated into higher genus and higher dimensions via the methods of~\cite{JR17}, which we further develop here.  

The first ``atomic'' superabundance phenomenon occurs at $(1,1)$, and its propagation explains all genus-one superabundance. Realizability in this case is characterized by a well-known condition discovered by Speyer, called \emph{well-spacedness}, see~\cite{R16,RSW17B,Sp07}. Speyer's work has implications for the enumerative geometry of elliptic curves~\cite{LR15,RSW17B}, and the higher-genus propagation of well-spacedness has applications in Brill–Noether theory~\cite{JR17}.  

In genus two and higher, prior to this paper, there had been no example of a condition guaranteeing the realizability of superabundant tropical curves, except those arising from genus one by propagation. A reason for this lack of progress is that existing proofs in genus one rely on features of elliptic curves—such as non-archimedean uniformization or the classification of elliptic singularities—that do not have straightforward generalizations to higher genus~\cite{RSW17B,Sp07}.

\subsection{Main results} The contributions of this paper are as follows:

\begin{enumerate}[(i)]
\item Qualitatively new conditions guaranteeing the realizability of balanced functions $\Gamma\to\RR$ on tropical curves of genus $2$ with contracted genus two subgraphs. 
\item A new proof that a balanced function $\Gamma\to\RR$ on a tropical curve of genus $1$ is realizable if and only if it is well-spaced. 
\item A new proof that non-superabundant functions $\Gamma\to\RR$ are realizable in any genus. 
\end{enumerate}

In Section~\ref{sec: bootstrap}, we develop a new dimensional induction that allows us to deduce partial results about maps to $\mathbb R^r$ using the natural projections to $\RR$. Consequently, results (ii) and (iii) imply the full well-spacedness result in all dimensions for genus one~\cite{RSW17B,Sp07}, as well as the corresponding result for non-superabundant maps in all dimensions~\cite{CFPU}. This approach also realizes many superabundant maps to $\RR^r$ in genus two.  

The geometric input is that the Hurwitz space of covers of $\mathbb P^1$ admits two natural moduli-theoretic compactifications: the space of logarithmic (or relative) stable maps~\cite{Che10,Li01} and the space of admissible covers~\cite{ACV,HM82}. We analyze the combinatorics of the natural comparison map in key base cases, and deduce the main results by combining these local analyses with global moduli arguments.

\subsubsection{Genus two} We work in characteristic $0$ for simplicity. An abstract tropical curve $\Gamma$ of {\it type $\Theta$} is a genus two abstract tropical curve that contains a subgraph homeomorphic to a $\Theta$ graph. Such a subgraph is necessarily unique; we call it the {\it core} and denote it by $\Theta$. The three edges of $\Theta$ will be called the {\it edges of the $\Theta$}.  

Consider a piecewise linear function $F\colon \Gamma \to \RR$ that is constant of value $0$ on the core. In the connected component of $F^{-1}(0)$ containing $\Theta$, the points that are incident to non-contracted edges in $\Gamma$ are called {\it critical points}. There is a unique path from each critical point to $\Theta$, which we call a {\it critical path}. The intersection of a critical path with $\Theta$ is called its {\it location}.

\begin{customthm}{A}\label{thm: genus-2-A}
Let $F\colon \Gamma\to\RR$ be a balanced function on an abstract tropical curve of $\Theta$ type, and assume that $F$ contracts the core. If 
\begin{enumerate}[(i)]
\item for each edge $E$ of $\Theta$, there exists a critical path of minimal length whose location is on $E$, and 
\item the lengths of all other critical paths are sufficiently large compared to the minimal length,
\end{enumerate}
then $F$ is realizable. 
\end{customthm}

In Figure~\ref{fig: example-C}, we show a balanced function of this kind. Except for the non-constructive nature of (ii), the result above is exactly parallel to the $r = 1$ case of the well-spacedness condition.

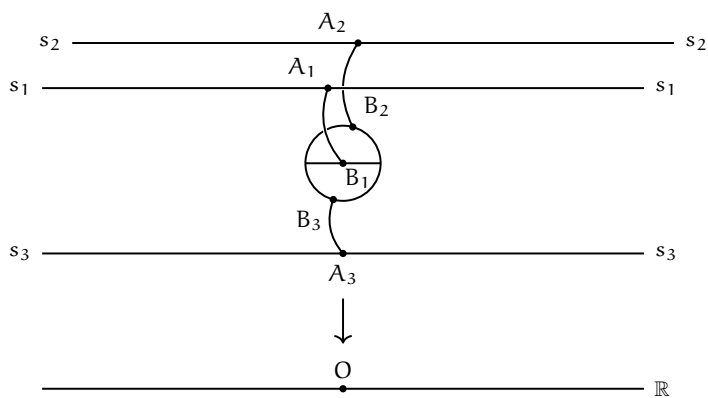
\begin{figure}[h!]
\begin{tikzpicture}
\draw (0.2,1.6) to [bend right] (75:0.5);
\draw[white,ultra thick] (-4,1) -- (4,1);
\draw (-4,1) node[anchor=east,font=\scriptsize] {$s_1$}  -- (4,1) node[anchor=west,font=\scriptsize] {$s_1$};
\draw (-3.6,1.6) node[anchor=east,font=\scriptsize] {$s_2$} -- (4.4,1.6) node[anchor=west,font=\scriptsize] {$s_2$};
\draw (0,0) circle (0.5);
\draw (-4,-1.2) node[anchor=east,font=\scriptsize] {$s_3$} -- (4,-1.2) node[anchor=west,font=\scriptsize] {$s_3$};
\draw[white,ultra thick] (-0.2,1) to [bend right] (0,0);
\draw (-0.2,1) to [bend right] (0,0);
\draw (255:0.5) to [bend right] (0,-1.2);
\draw (0.5,0) -- (-0.5,0);
\filldraw[black] (-0.2,1) circle (1pt) node[anchor=south east,font=\footnotesize] {$A_1$};
\filldraw[black] (0.2,1.6) circle (1pt) node[anchor=south east,font=\footnotesize] {$A_2$};
\filldraw[black] (0,-1.2) circle (1pt) node[anchor=north,font=\footnotesize] {$A_3$};
\filldraw[black] (0,0) circle (1pt);
\node[font=\footnotesize] at (0.2,-0.2) {$B_1$};
\filldraw[black] (75:0.5) circle (1pt) node[anchor=south west,font=\footnotesize] {$B_2$};
\filldraw[black] (255:0.5) circle (1pt) node[anchor=north east,font=\footnotesize] {$B_3$};
\draw[->] (0,-1.8) -- (0,-2.4);
\draw (-4,-3) -- (4,-3) node[anchor=west,font=\scriptsize] {$\mathbb{R}$};
\filldraw[black] (0,-3) circle (1pt) node[anchor=south,font=\footnotesize] {$O$};
\end{tikzpicture}
\caption{A balanced map with a contracted $\Theta$ and three attaching points. It is realizable provided the lengths of the three segments $A_iB_i$ are all equal.}\label{fig: example-C}
\end{figure}

Our next result is qualitatively different from the genus $1$ case. Observe that the $\Theta$ graph admits a unique involution whose quotient is a tree, given by flipping each edge. Pairs of points that are exchanged by the involution are {\it conjugate}. 

\begin{customthm}{B}\label{thm: genus-2-B}
Let $F\colon \Gamma\to\RR$ be a balanced function on an abstract tropical curve of $\Theta$ type, and assume that $F$ contracts the core. If 
\begin{enumerate}[(i)]
\item there exists an edge of $\Theta$ and two critical paths of minimal length whose locations are conjugate, and
\item the lengths of all other critical paths are sufficiently large compared to the minimal length, 
\end{enumerate}
then $F$ is realizable. 
\end{customthm}

In Figure~\ref{fig: example-D}, we depict an example of a balanced function of this kind. In Proposition~\ref{prop: genus-2-Weierstrass} we show an analogous result where the Weierstrass points, i.e. fixed points of the hyperelliptic involution, play a crucial role. 

The moduli of maps from $\Theta$ curves to $\RR$ that contract the core has excess dimension $2$ compared to the expected dimension, and so one expects realizability to be a codimension $2$ condition. In these two theorems, we see that the two conditions can arise from both the locations of critical paths and their lengths. 

In the main text, we include additional such results, including when the core is a dumbbell rather than $\Theta$. We focus on the combinatorics of the $\Theta$ here, largely for expository reasons. 

\begin{figure}[h!]
\includegraphics{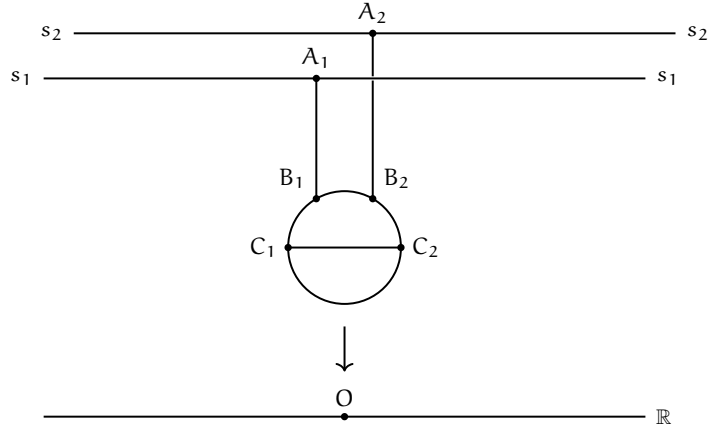}
\caption{A balanced map with a contracted $\Theta$ and two attaching points that are conjugate. }\label{fig: example-D}
\end{figure}

\subsubsection{Hurwitz modifiability and combinatorial rederivations} Let $F\colon \Gamma\to \RR$ be a balanced function. It is known by work of Amini--Baker--Brugall\'e--Rabinoff that the realizability of $F$ is equivalent to the existence of, what we dub here, an {\it $H$-modification}~\cite{ABBR1,ABBR2}. In Section~\ref{sec: modifications}, we re-explain this using the geometry of relative maps and admissible covers. Either way, the upshot is that realizability is equivalent to the existence of a commutative diagram
\[
\begin{tikzcd}
\hat \Gamma\arrow[swap]{d}{\hat F}\arrow{r} &\Gamma\arrow{d}{F}\\
\hat T\arrow{r} & \RR,
\end{tikzcd}
\]
where the horizontal arrows are tree attachments, the map $\hat F$ is finite and preserves balanced functions under pullback, and at every vertex, the local Hurwitz problem has a solution. 

The $H$-modifiability of a balanced function might seem like a basic tool for studying realizability, but both the graph-theoretic aspects and the Hurwitz problem are of high combinatorial complexity, and few general techniques apply to either. One motivation for this paper was to investigate whether $H$-modifiability can be used in practice. The answer appears to be positive.  

We now examine what these techniques imply for the existing body of work on tropical realizability~\cite{CFPU,Mi03,NS06,Sp07}. The methods used in these results are very different from $H$-modifiability. Approaching the problem from the perspective of $H$-modifiability, however, leads naturally to nontrivial combinatorial statements. For example: {\it (i) every balanced function on a trivalent tropical curve that does not contract a cycle is $H$-modifiable,} and {\it (ii) well-spacedness for a balanced function on a trivalent genus-$1$ tropical curve is equivalent to $H$-modifiability.}

Our second main contribution is to provide direct proofs of these combinatorial statements.  

\begin{customthm}{C}
Let $\Gamma$ be an abstract, trivalent tropical curve, and let $F\colon \Gamma \to \RR$ be a balanced map. Assume that $F$ does not contract a cycle. Then $F$ is $H$-modifiable. As a consequence, such maps are realizable.
\end{customthm}

This result reproduces a theorem of Cheung--Fantini--Park--Ulirsch~\cite{CFPU} specialized to target dimension $1$, and, with the dimensional reduction technique developed in Section~\ref{sec: bootstrap}, provides a new proof of their results.  

We now turn to genus $1$. Speyer's ``well-spacedness'' condition~\cite{Sp07} gives a combinatorial criterion for realizability. In the trivalent case, it asserts that a genus-$1$ tropical curve $\Gamma$ with a balanced map $F\colon \Gamma \to \RR$ that contracts the core is realizable if and only if there are at least two critical paths of minimal length.  

\begin{customthm}{D}\label{thm: well-spacedness}
Let $\Gamma$ be an abstract, trivalent, genus-$1$ tropical curve, and let $F\colon \Gamma \to \RR$ be a balanced map that contracts a cycle. Then $F$ is $H$-modifiable if and only if $F$ is well-spaced. Consequently, such a map is realizable if and only if it is well-spaced. 
\end{customthm}

We discuss the non-trivalent case in the main text. The result provides an independent proof of Speyer's realizability theorem~\cite[Theorem~3.3]{Sp07} for balanced functions. The bootstrapping technique of Section~\ref{sec: bootstrap} implies the higher dimensional case from the above statement. 

\subsection{Outlook and techniques} 
Our approach to realizability is more elementary than earlier ones, which rely on heavier technical machinery such as Tate's non-archimedean uniformization for elliptic curves and the residue theory of elliptic singularities. These techniques have not yet been extended to higher genera, but our results hint at the subtleties involved in this. The cost of our method is high combinatorial complexity: finding appropriate tropical modifications and verifying the relevant Hurwitz conditions is challenging in practice, and much of the paper is devoted to these combinatorial constructions. Once such techniques are developed, it seems potentially fruitful to combine them with other viewpoints.  

One promising direction is the connection to the singularity theory of curves. Our methods are particularly effective for producing examples, whereas the geometry of Gorenstein curves used in~\cite{RSW17B} seems better suited for proving statements once they have been conjectured. We expect that progress in genus $2$ and beyond will involve revisiting the genus-$2$ results using Abel--Jacobi maps of Gorenstein curves of genus $2$, following Battistella--Carocci~\cite{BC23}.\footnote{During the conference {\it GeNeSys} in Belalp in September 2024, Carocci sketched why these results are consistent with the perspective from genus-$2$ curve singularities.}

We use some general algebro-geometric techniques -- essentially ``pure thought'' arguments -- and we mention them here in case they may be useful. First, while we use results on modifiability~\cite{ABBR2}, we explain this as a comparison between two moduli spaces of maps from curves. The basic object for us is a map $C\to\PP^1$ over a non-archimedean field. We view this as a point in the Hurwitz space and produce tropical pictures from the combinatorics of limits in two different compact moduli spaces: the admissible cover space~\cite{ACV,HM82} and the logarithmic stable map space~\cite{Che10,Li01}:
\[
\mathsf{AdmCovers}\supset\mathsf{HurwCovers}\subset\mathsf{LogStMaps}. 
\]
The tropical realization problem is a type of {\it smoothing} problem for degenerate maps. But $\mathsf{AdmCovers}$ is logarithmically smooth, so this problem always has a solution. Therefore, a smoothable point of $\mathsf{LogStMaps}$ must be ``promotable'' to a point of $\mathsf{AdmCovers}$, and this is the modifiability condition. It suggests a general picture: the interacting combinatorics of different semistable reduction theorems create constraints on tropical realizability. 

Second, as mentioned above, we explain how to deduce results about realizing maps to higher dimensional tori from maps to lower dimensional tori. This is based on a lifting argument for tropical intersections. In genus $0$, $1$, and in the non-superabundant case, this is sufficient to understand the problem completely. In higher genus, this is not the case, as shown by Koyama's recent results~\cite{Koy23}. The target dimension $2$, genus $2$ situation in complete generality requires new ideas. 

And third, the non-constructive aspect of our genus $2$ results -- the fact that lengths have to be ``very long'' -- seems to provide a useful conclusion that is relatively cheap to prove. It uses the structure of compactified tropicalizations to deduce the existence of points in the interior of the moduli space by constructing points at infinity in the combinatorial picture. 

We conclude with a word about applications. The present results in genus $2$ can be leveraged, using the genus induction of~\cite{JR17}, to prove regeneration theorems for linear series on tropical curves, and we believe this should have applications in Brill--Noether theory. In particular, one can hope to use this to understand the codimension of the Brill--Noether locus when the Brill--Noether number is negative~\cite{JR17,Pfl23}.  

We expect other applications, for example, to the study of fixed complex structure enumerative geometry in genus $2$, following~\cite{LR15,Z03}. In a different direction, Hicks has shown that the realizability of tropical curves has a surprising link, via mirror symmetry, to the construction of Floer theoretically unobstructed Lagrangian submanifolds in the torus~\cite{Hicks22}. For this link to Lagrangian geometry, it is important that the tropical curves are embedded, or at least immersed. Interestingly, this makes the dimensional reduction in Section~\ref{sec: bootstrap} something on the algebraic side that might not be available on the mirror.

\subsection*{Acknowledgements} The question pursued here has its origins in an uncompleted project of R. Cavalieri, H. Markwig, and the third author dating back to 2014, concerning the tropicalization of the double ramification cycle~\cite{CMR14a}. We are grateful for many years of conversations with them on these and related topics. Our thanks also go to D. Jensen for encouraging us to return to the question, as well as to L. Battistella, F. Carocci, and J. Wise for inspiring discussions on realizability and Gorenstein singularities. Finally, we are very grateful to K. Christ, E. Katz, J. Hicks, S. Koyama, S. Payne, I. Tyomkin, and D. Speyer for their helpful conversations. The manuscript was improved thanks to the input of an anonymous referee. 

\noindent
ML was supported by the Philippa Fawcett program. ML and CKN conducted the research here under the Cambridge SRIM program in 2023. DR was supported by the EPSRC New Investigator Grant EP/V051830/1 and the EPSRC Frontier Research Grant EP/Y037162/1.

\section{Tropical covers and modifications}\label{sec: modifications}

We recall the basics of tropical curves and maps between them, following~\cite{ABBR1,ABBR2,CMR14a}.  

\subsection{Tropical curves}\label{sec: tropical-curves} A {\it tropical curve} is a connected, finite metric graph with unbounded rays and a genus labeling at its vertices. Precisely, it is an enhancement of an abstract connected graph $G$ by
\begin{enumerate}[(i)]
\item a length function $\ell\colon E(G)\to\QQ_{>0}$,
\item a finite set of {\it rays} attached to the vertices by $m\colon [n]\to V(G)$, and
\item a genus function $\hat g\colon V(G)\to \mathbb Z_{\geq 0}$.
\end{enumerate}
Given this data, we form a non-compact metric space as follows. Identify the edge $e$ with an interval of length $\ell(e)$ to form a metric graph $\Gamma_0$. We form $\Gamma$ by
\[
\Gamma = \left(\Gamma_0\sqcup \coprod_{i\in[n]}\RR_{\geq 0}\right)/\sim,
\]
where the equivalence relation identifies each $0$ in the disjoint union with the point corresponding to the vertex $m(i)$. The unbounded rays are referred to as {\it rays}. 

Of course, one can also study {\it disconnected} tropical curves. The connected case is better for most of our exposition, so when disconnected curves arise, we explicitly use the phrase ``possibly disconnected''.
We typically work in the case where $\hat g$ is everywhere equal to $0$. This is sometimes referred to as an {\it explicit} tropical curve. The only exception to this is in Section~\ref{sec: tropicalizations}, which discusses tropicalization via semistable reduction. The exposition of this is more natural with an allowance for vertex genera.

Note that it is sometimes natural to allow the edge lengths to take real values, rather than only rational values. However, the space of realizable tropical maps, whose analysis is our main goal, is a {\it rational} polyhedral subcomplex of the moduli space of all tropical maps~\cite{R16}. Realizability results in this more general context, therefore, follow from the results that we discuss here. 

A metric graph $\Gamma$ comes with a distinguished set of {\it piecewise linear functions with integral slopes}, or {\it piecewise linear functions} for short. These are continuous functions
\[
\Gamma\to\RR
\]
that restrict to linear functions with integer slopes on the edges and rays. 

Given two metric graphs $\Gamma'$ and $\Gamma$, we study {\it piecewise linear maps} between them. These are continuous maps $F\colon\Gamma'\to \Gamma$ that induce a graph map, with the further condition that, when restricted to every edge or ray of $\Gamma'$, the map $F$ is linear with an integer slope.   

For ease of terminology, given a vertex $v$ on $\Gamma$, we refer to flags consisting of a vertex $v$ and an edge or ray incident to it as the {\it tangent directions} at $v$. The ``trivial'' flag $(v,v)$ is formally viewed as the $0$-tangent direction. Given a vertex $v$ and an edge or ray $e$ to which it is incident, we refer to the flag $(v,e)$ as the ``direction associated with $e$''. We visualize these as germs of edges or rays starting at $v$, moving in the direction of one of the flags. 

Given a piecewise linear map $F\colon \Gamma'\to \Gamma$, every tangent direction based at $v'$ in $\Gamma'$ maps to a tangent direction based at $F(v')$. It is the $0$ direction when the edge or ray that determines the tangent direction is contracted to $F(v')$. 

A piecewise linear function has a well-defined integral slope along a tangent direction, by convention directed {\it away} from $v$. Given a map $\Gamma'\to\Gamma$, we can measure the slope along any tangent direction in $\Gamma'$. Indeed, a tangent direction of $\Gamma'$ maps to a tangent direction of $\Gamma$, locally as an integer linear map. If $(v',e')$ is a tangent direction on $\Gamma'$, and if $e'$ maps to the vertex $v$, the slope is equal to $0$.

\subsection{Balanced functions and harmonic maps} Piecewise linear functions and piecewise linear maps will arise for us as tropicalizations of meromorphic functions on Riemann surfaces and maps between Riemann surfaces, respectively. However, not all such maps are relevant to geometry. The {\it balancing condition} identifies a subset of such functions. 
\begin{definition}
A piecewise linear function $F\colon \Gamma \to \mathbb{R}$ is called \emph{balanced} or \emph{harmonic} if, at every vertex $v$ of $\Gamma$, the sum of the slopes of $F$ along the tangent directions at $v$ is $0$. 
\end{definition}

Balanced functions are the tropical analogues of rational functions on algebraic curves. The corresponding notion for maps between tropical curves is recorded next.

\begin{definition}
Let $F\colon \Gamma' \to \Gamma$ be a piecewise linear map of tropical curves. Let $v'$ be a vertex of $\Gamma'$ with image $v = F(v')$, and let $e$ be an edge or ray of $\Gamma$ based at $v$. The \emph{local degree of $F$ at $v'$ in the direction of $e$} is the sum of the slopes of $F$ along all tangent directions $(v', e')$ in $\Gamma'$ that map to the direction $(v, e)$.

The map $F$ is said to be \emph{locally of pure degree at $v'$} if, for any edge $e$ based at $v$, this local degree is independent of the choice of $e$. 
\end{definition} 

\begin{definition}
A map $F \colon \Gamma' \to \Gamma$ is \emph{harmonic} if it is surjective and locally of pure degree at every vertex $v'$ of $\Gamma'$. We usually denote this degree by $d_{v'}(F)$, or simply $d_{v'}$ when it is clear from context.
\end{definition}

In other words, given a map $\Gamma' \to \Gamma$, the local degree at a vertex $v'$ of $\Gamma'$ can be measured \emph{after choosing} a tangent direction at its image $v = F(v')$ in $\Gamma$. This measurement is given by the sum of the slopes along all directions in $\Gamma'$ based at $v'$ that map to the chosen direction in $\Gamma$. The \emph{purity} condition asserts that this choice of direction in the target does not affect the local degree. In practice, this is how the condition is used. The harmonic condition asserts that purity holds at every vertex of the domain $\Gamma'$.

The harmonic condition for a map $F\colon \Gamma' \to \Gamma$ can also be described more conceptually though less practically in terms of balanced functions. Specifically, a piecewise linear map is harmonic if and only if, for any open set $U \subset \Gamma$, the pullback of any balanced function $U \to \mathbb{R}$ is again a balanced function $F^{-1}(U) \to \mathbb{R}$. In particular, compositions of harmonic maps are harmonic. See~\cite[Section~2.9]{ABBR1} for further details.

\subsection{Mildness conditions} We introduce two conditions on harmonic maps that allow realizability to be analyzed in a straightforward manner. The discussion here assumes that all vertex genera are equal to $0$. The contents of this section do not require a connectivity hypothesis on the tropical curves. 

\begin{definition}
A harmonic map $F\colon\Gamma'\to \Gamma$ is {\it finite} if the slope of $F$ is nonzero on all edges and rays of $\Gamma'$. 
\end{definition}

Given a finite harmonic map of tropical curves, each vertex of the domain determines potential discrete data for a map between Riemann spheres. We record this as a construction below. 

\begin{construction}
Let $F\colon \Gamma'\to\Gamma$ be a finite harmonic map of tropical curves. Let $v'$ be a vertex of $\Gamma'$, and let $e'_1,\ldots,e'_r$ be the non-zero tangent directions at $v'$. Let $v$ be the image of $v'$ and let $e_1,\ldots,e_s$ be the non-zero tangent directions of $\Gamma$ based at $v$. There is a natural map
\[
\overline F\colon\{e'_1,\ldots,e'_r\}\to \{e_1,\ldots,e_s\}.
\]
The map $F$ induces a linear map from the tangent direction $e'_i$ to $\overline F(e'_i)$. We call the slope of this map $m_i$. 

From these data, we can extract {\it ramification data at $v'$} for a map of Riemann surfaces. Specifically, we may consider the existence of maps
\[
(\PP^1,p_1,\ldots,p_r)\to (\PP^1,q_1,\ldots,q_s),
\]
such that $p_i$ maps to $q_j$ if $\overline F(e'_i) = e_j$, and the ramification order at $p_i$ is $m_i$. We denote the set of maps with these data by $\mathsf R(F,v')$. 
\end{construction}

Given a finite harmonic map $F\colon \Gamma'\to \Gamma$, the set $\mathsf R(F,v')$ may be empty — for example, the conditions on ramification provided by Riemann--Hurwitz might obstruct the existence of such a map. Even if the Riemann--Hurwitz constraints are satisfied, the set of maps may still be empty. 

\begin{definition}\label{def: locally-realizable}
A finite harmonic map $F\colon \Gamma'\to\Gamma$ is {\it locally realizable} if for every vertex $v'$ in $\Gamma'$, the set $\mathsf R(F,v')$ is nonempty. 
\end{definition}

In the next subsection, we introduce tropicalization for maps between curves and the associated inverse problem. Locally realizable finite harmonic maps will turn out always to be realizable.

\subsection{Tropicalization for families of curves, functions, and covers}\label{sec: tropicalizations} We discuss tropicalization for maps between curves. The notion has been treated before: using Berkovich spaces, by Amini, Baker, Brugall\'e, and Rabinoff~\cite{ABBR1,ABBR2}, as well as using admissible covers and logarithmic mapping spaces in~\cite{CMR14a,R16}. All these papers prove much more than we need; however, they require additional technical machinery. We provide an elementary presentation tailored to our applications.

\subsubsection{First for curves} Let $K$ be a discretely valued field with valuation ring $R$ and residue field $\mathbb C$. Let $t\in R$ be a uniformizing parameter and let $\mathsf{val}$ denote the valuation on $K^\times$. 

Consider a smooth marked curve $(C,p_1,\ldots,p_r)$ of genus $g$ over $K$. We will assume that $2g(C)-2+r$ is positive. By the stable reduction theorem, or equivalently by the weak valuative criterion of properness for $\Mbar_{g,r}$, after possibly replacing $K$ with a finite field extension $K\subset K'$, there exists a unique stable model. Specifically, this is a family:
\[
(\mathcal C,p_1,\ldots,p_r)\to \Spec R',
\]
The special fiber $\mathcal C_0$ of this family is a stable, nodal, genus $g$ curve. The map $\Spec R'\to \Spec R$ is a finite surjective map, possibly ramified over its closed point. 

The total space $\mathcal C$ is a toroidal surface, with the boundary given by the special fiber, together with the closure of the $r$ sections. This means that the total space is smooth away from the nodes of $\mathcal C_0$. At the nodes, the space is allowed to have singularities of $A_{n-1}$-type, given locally by
\[
xy = \alpha\cdot s^{n},
\]
where $\alpha$ is a unit in $R'$, and $s$ is a uniformizer in $R'$. 

By convention, we extend the valuation $\mathsf{val}\colon K^\times\to \ZZ\subset \RR$ to $K'$, in such a way that the valuation of the parameter $t$ is still $1$, and the valuation of $s$, the new uniformizer in $R'$, lies in $\mathbb Q$. In fact, it will be equal to $\frac{1}{m}$ where $m$ is the order of the field extension $K\subset K'$. 

The {\it tropicalization} of $(C,p_1,\ldots,p_n)$ is a tropical curve denoted $\Gamma(C)$, $\Gamma(\mathcal C)$, or simply $\Gamma$ when the context is clear, determined as follows:
\begin{itemize}
\item The abstract graph is the dual graph of $\mathcal C_0$, with genus decoration at the vertices.
\item The marking function $[n]\to V(G)$ sends $i$ to the vertex whose associated component contains $p_i$. 
\item Let $e$ be an edge with associated node $q_e$. If $q_e$ is an $A_{k-1}$ singularity in $\mathcal C$, then the length of the edge is $k$. 
\end{itemize}

The cone over this tropical curve $\Gamma$ is a generalized cone complex $\Sigma(\mathcal C)$, see~\cite{ACP}. The valuation on $R'$ gives rise to a map
\[
\Sigma(\mathcal C)\to \RR_{\geq 0}.
\]
Recall that the valuation of $s$ is an element of $\mathbb Q_{\geq 0}$. The fiber of the map above over the point $\mathsf{val}(s)$ is canonically identified with $\Gamma$, and enhances it with the structure of a metric graph with rational edge lengths.

\subsubsection{...then with a function...}\label{sec: trop-rsm} We add a meromorphic function to the picture. Let $(C,p_1,\ldots,p_r)$ be a curve over a discretely valued field $K$. Consider a rational map
\[
\begin{tikzcd}
C \arrow{d}\arrow[dashed]{r}{\varphi} &\mathbb G_m\\
\Spec K, &
\end{tikzcd}
\]
whose indeterminacy is contained in the set of marked points $\{p_1,\ldots,p_r\}$. The added structure of the map to $\mathbb G_m$ enhances the tropical curve $\Gamma$ with a balanced function
\[
F\colon \Gamma\to\RR.
\]
The slope of $F$ along the ray corresponding to $p_i$ equals the order of the zero or pole of $\varphi$ at the point $p_i$ on $C$. 

The map $F$ can be constructed using the theory of harmonic functions on Berkovich spaces~\cite{BPR16}. Alternatively, we can use the semistable reduction theorem for maps to the pair $(\mathbb P^1|0+\infty)$, as worked out by J. Li~\cite{Li01}. The output is that, after possibly replacing $K$ with a field extension, there is a diagram
\[
\begin{tikzcd}
\mathcal C\arrow{dr}\arrow{rr}{\psi} & & \mathcal P\arrow{dl}\\
&\spec R',&
\end{tikzcd}
\]
completing the map $\varphi$ above, and further: (i) where $\mathcal C$ is a semistable family of genus $g$ curves with general fiber $C$, (ii) the space $\mathcal P$ is a semistable family of $2$-pointed genus $0$ curves, with special fiber a chain of $\mathbb P^1$'s, and (iii) the preimages of nodes/markings of $\mathcal P_0$ are nodes/markings of $\mathcal C_0$. 

The conditions (i)--(iii) are exactly the conditions that guarantee that $\mathcal C_0\to\mathcal P_0$ is a {\it relative stable map} in the sense of Li.\footnote{In principle, one should also impose the kissing condition that the ramification orders on the preimages of a node under normalization are equal, as in~\cite{Li01}. However, this is automatic for smoothable maps, and our maps are smoothable by construction.} The surfaces $\mathcal C$ and $\mathcal P$ are again toroidal surfaces with at worst $A_k$-singularities. The toroidal structure on $\mathcal P$ is given by the special fiber, together with two sections, corresponding to the closure of $0$ and $\infty$ in $\mathbb P^1_K$. By construction, the preimage of the toroidal boundary of $\mathcal P$ is contained in that of $\mathcal C$, so there is an induced map of cone complexes $\Sigma(\mathcal C)\to \Sigma(\mathcal P)$. However, $\Spec R$ is also a toroidal pair with boundary given by the closed point and the cone complex $\RR_{\geq 0}$. Since the arrows in the triangle above have the property that the preimage of the boundary is contained in the toroidal boundary, the maps are toroidal, and by functoriality results for tropicalizations~\cite{ACP,U15}, we obtain a triangle of cone complexes:
\[
\begin{tikzcd}
\Sigma(\mathcal C)\arrow{rr}\arrow{dr} & & \Sigma(\mathcal P)\arrow{dl}\\
&\RR_{\geq 0}.&
\end{tikzcd}
\]
Again, the fiber of $\mathcal P$ over the valuation of the uniformizer in $R'$ is a canonically polyhedral decomposition of $\RR$. The fiber produces a map
\[
F\colon \Gamma'\to \RR.
\]
We note that it is possible that, when viewed as the fiber over $\Sigma(\mathcal C)$, the graph $\Gamma'$ contains vertices that are $2$-valent and whose corresponding components have $0$ genus. However, by balancing, it is straightforward to check that $F$ descends to a piecewise linear function after coarsening $\Gamma'$ to remove these $2$-valent vertices from the graph structure, i.e., ``erasing'' the vertex. This is precisely the graph $\Gamma$ defined by the stable limit in the previous discussion.

We call the resulting balanced function 
\[
\Gamma\to \RR
\]
the {\it tropicalization} of the data $(C,p_1,\ldots,p_r,\varphi)$. 

\subsubsection{...and finally for covers} In the discussion above, we obtain from the meromorphic function on a curve $(C,\varphi)$ a map
\[
C\to \PP^1
\]
of {\it pairs} -- where $C$ is marked at the roots and poles of $\varphi$ and $\mathbb P^1$ is marked at $0$ and $\infty$. The resulting semistable reduction theorem for maps of pairs~\cite{Li01} gave rise to the tropical curve and balanced function. 

Going a step further, we can {\it in addition} mark all the branch points of $\varphi$ on $\mathbb P^1$, and then mark all the preimages of these points in $C$. We denote the marked points on $\mathbb P^1$ by $b_1,\ldots, b_m$ and the points on the domain $p_{i1},\ldots,p_{ik_i}$, where $\varphi$ sends $p_{im}$ to $b_i$ for all $m$. We now view $\varphi\colon C\to\PP^1$ as a map of pairs, with all ramification and branch points (and possibly more) marked. 

Now suppose $\varphi\colon C\to \PP^1$ is as above, but defined over a field $K$ equipped with a discrete valuation, inducing a valuation subring $R$. We may again apply semistable reduction for maps of pairs. The results of Li apply again in this setting, but in fact, once all ramification is marked, this follows from earlier work. Specifically, we may use the semistable reduction theorem for admissible Hurwitz covers of nodal curves, due to Harris--Mumford~\cite{HM82}. To summarize the output, possibly after replacing $K$ with a finite extension, we have a map
\[
\begin{tikzcd}
\mathcal C\arrow{dr}\arrow{rr}{\psi} & & \mathcal P\arrow{dl}\\
&\spec R,&
\end{tikzcd}
\]
where $\mathcal C$ is a semistable model for $C$ and $\mathcal P$ is a semistable model for $\mathbb P^1$. The special fiber of $\mathcal P$ is a tree, but not necessarily a chain, of smooth rational curves. 

Proceeding as in the previous case, we obtain a triangle of cone complexes
\[
\begin{tikzcd}
\Sigma(\mathcal C)\arrow{rr}\arrow{dr} & & \Sigma(\mathcal P)\arrow{dl}\\
&\RR_{\geq 0}.&
\end{tikzcd}
\]
If we pass to the fibers over $1$, we obtain a map
\[
F\colon \Gamma\to T
\]
where $\Gamma$ is a tropical curve and $T$ is a tropical curve whose underlying graph is a tree.

\begin{proposition}
The map $F\colon \Gamma\to T$ is finite and harmonic. 
\end{proposition}

\begin{proof}
The proposition follows immediately from the definition of an admissible cover, and in particular, from the fact that the limiting admissible cover is a finite map -- see~\cite[Section~3]{CMR14a}.
\end{proof}

In all these constructions, we ultimately obtain a tropical curve with possibly nonzero genus decorations at the vertices. But we flag again that in our applications, which concern {\it reversing} this process, we will start with tropical curves whose vertex decorations are everywhere equal to $0$. 

\subsection{Hurwitz modifiability} In this section, we explain the main technical theorem on which our results and examples rest. We begin with the strategy. 

\subsubsection{Basic strategy} Let $F\colon \Gamma\to\RR$ be a balanced function. Our goal is to determine if it arises via the procedure described in Section~\ref{sec: trop-rsm}. In a moment, we will explain a procedure that modifies harmonic maps of graphs. 

The idea is that these modifications will make $F$ closer to being {\it finite} harmonic. If this is achieved completely, and if one can independently verify local realizability of Definition~\ref{def: locally-realizable}, then one can appeal to the deformation theory of admissible covers to establish realizability. 

Conversely, if we have realizability of $F\colon \Gamma\to\RR$, the semistable reduction theorem that establishes the properness of the space of admissible covers will guarantee that such modifications to locally realizable finite harmonic maps exist. We emphasize that the result itself is known~\cite{ABBR1,ABBR2}, see also~\cite{UZ19}; however, we believe the presentation here may be valuable. 

\subsubsection{Modifications and $H$-modifiability} We introduce the key combinatorial move in this paper. 

\begin{definition}
Let $\Gamma$ be a tropical curve, and let $p$ be a point on the underlying metric space. The {\it modification of $\Gamma$ at $p$}, denoted $\Gamma(p)$, is the graph obtained by 
\[
\Gamma(p) = \left(\Gamma\sqcup \RR_{\geq 0} \right)/\sim,
\]
where the equivalence relation identifies the origin in the ray with $p$ on the graph $\Gamma$. A {\it tropical modification $\Gamma_1$ of $\Gamma$} is the tropical curve obtained by a finite sequence of modifications at points. Note that there is an induced harmonic map
\[
\Gamma_1\to \Gamma,
\]
by contracting the new rays to their base point. 
\end{definition}

Note that modifications are local on $\Gamma$, so they apply equally well in the disconnected setting. 

\begin{remark}
If $\Gamma_1\to \Gamma$ is a modification, we can canonically view 
\[
\Gamma \subset \Gamma_1
\]
and we will frequently use this.    
\end{remark}
\begin{remark}[Geometric meaning]
Modifications have geometric meanings. A tropical curve $\Gamma$ arises from a pointed smooth curve over a discretely valued field $K$. By the semistable reduction theorem, we can find a family (after passing to a field extension) over the valuation ring $R$ with nodal special fiber, where the limits of the marked points are distinct smooth points, and the only singularities of the total space are possibly $A_n$-singularities at the nodes. The total space of this family of curves has two natural sets of divisors – those corresponding to the marked points and those corresponding to the special fiber components. The limits of the marked points are not allowed to pass through the nodes in the special fiber and must remain distinct. From this, we build $\Gamma$ as discussed above. 

The simplest type of modification is obtained by marking a new point in the general fiber, possibly after passing to a finite extension of $K$. When we take the limit of this point in the special fiber, it may (i) pass through a node or (ii) pass through another marking. However, the semistable reduction theorem tells us that we can find a new limit where these issues are resolved, and this leads to a new tropical curve $\Gamma'$. Tropically, this is precisely the procedure that produces a modification by adding a single ray. In case (i) the ray is based in the interior of an edge, and in (ii) it is based on another ray. If the new marked point limits to an unmarked point of a component, the new ray is based at a vertex. 

We can repeat this process, of course, and this leads to the general description of a modification.
\end{remark}

\begin{definition}
Let $F\colon\Gamma'\to \Gamma$ be a harmonic map of tropical curves. A {\it modification of $F$} consists of 
\begin{enumerate}[(i)]
    \item tropical modifications $\Gamma'_1$ and $\Gamma_1$ of $\Gamma'$ and $\Gamma$, respectively, and
    \item a harmonic map $\Gamma'_1\to\Gamma_1$ fitting into a commutative square
    \[
    \begin{tikzcd}
    \Gamma'_1\arrow{r}\arrow{d} & \Gamma'\arrow{d}\\
    \Gamma_1\arrow{r} & \Gamma.
    \end{tikzcd}
    \]
\end{enumerate}
We say that $F$ is {\it $H$-modifiable} if there exists a modification of $F$ that is finite and locally realizable. 
\end{definition}

Modifications are our key tool in analyzing realizability, and this relies on the following: 

\begin{theorem}
A harmonic map $F\colon\Gamma'\to \Gamma$ of tropical curves with all vertex genera equal to $0$ is realizable if and only if it is $H$-modifiable.
\end{theorem}

\begin{proof} We divide the proof into two parts: going from the $H$-tropical modification to a realization, and going from a realization to an $H$-tropical modification. 

\noindent
{\sc I. Modification to realization.} Fix a harmonic map as stated in the theorem. Let 
\[
\begin{tikzcd}
\Gamma'_1 \arrow{r} \arrow{d} & \Gamma' \arrow{d}\\
\Gamma_1 \arrow{r} & \Gamma
\end{tikzcd}
\]
be an $H$-modification. By the local realizability aspect of the definition of an $H$-modification find marked rational curves $C_v$ and $C_u$, for each vertex $v$ of $\Gamma'_1$ with image $u$ under $F$, whose markings are respectively induced by the edges incident to $v$ and $u$, as well as a map
\[
C_v \to C_u
\]
determined by $\Gamma'_1 \to \Gamma_1$. Specifically, a marking of $C_v$ maps to a marking of $C_u$ precisely when the corresponding edge of $\Gamma'_1$ maps to the appropriate edge in $\Gamma_1$. Moreover, the ramification will be equal to the slope at the corresponding edge. 

By doing this for all $v$ and gluing the curves as dictated by the graphs, we obtain an admissible cover
\[
\overline C'_1 \to \overline C_1
\]
of nodal curves. We claim the curves can be simultaneously smoothed along with the map to obtain $\mathcal C'_1 \to \mathcal C_1$ over a discrete valuation ring $R$ that is ``compatible with the lengths'' in the following sense: for each edge $e$ in $\Gamma'_1$ resp. $\Gamma_1$, the length of that edge is equal to the valuation of the deformation parameter of the node of the curve $\mathcal C'_1$ resp. $\mathcal C_1$. 

The claim follows from the smoothness of the moduli space of (orbifold) admissible covers~\cite{ACV,HM82}. To see this, note that $[\overline C'_1 \to \overline C_1]$ defines a moduli point in the space of admissible covers. The moduli space is well-known to be smooth. In an \'etale neighborhood $U$ of this point, to each node $q$ of the base curve $\overline C_1$ corresponds a hypersurface $V_q$ parameterizing deformations of this cover where the node persists. If we choose an equation $f_q$ for each $V_q$, the parameters $f_q$ together form a subset of a local system of coordinates for $U$. 

Now, given a map $\spec R \to U$, the tropicalization of the corresponding family of covers can be identified combinatorially, that is, ignoring the lengths of the edges, with $\Gamma'_1 \to \Gamma_1$. The edge lengths are equal to the valuations of the parameters $f_q$ – these give the lengths on edges of the base graph, and the lengths on the source are determined since the slopes are known. In particular, by smoothness, we can find a map $\spec R \to U$ whose corresponding family of covers realizes the given $\Gamma'_1 \to \Gamma_1$. 

It remains to produce the realization as the original $\Gamma' \to \Gamma$. The maps $\Gamma'_1 \to \Gamma'$ resp. $\Gamma_1 \to \Gamma$ determine forests of rational curves on $\overline C'_1$ resp. $\overline C_1$. We can contract the corresponding surfaces to $\mathcal C'$ and $\mathcal C$. Since there is a map $\Gamma' \to \Gamma$, there is a map of surfaces
\[
\mathcal C' \to \mathcal C.
\]
Tropicalizing these families, following Section~\ref{sec: trop-rsm}, we obtain a realization of $\Gamma' \to \Gamma$, as promised. This completes the first part of the proof. 

\noindent
{\sc II. Realization to modification.} Suppose we have a realization. Precisely, suppose we have a family $\mathcal C' \to \mathcal C$ over a valuation ring $R$. Assume the generic fiber of this map is finite, that is, a Hurwitz cover, with tropicalization $\Gamma' \to \Gamma$. 

Let $C' \to C$ be the general fiber, defined over the valued field $\spec K$. By properness of the moduli space of admissible covers, after possibly passing to a finite extension (which we suppress from the notation), we can find a new limit – the limit as an admissible cover. Denote these families and the map by:
\[
\mathcal C'_1 \to \mathcal C_1.
\]
Note that this map is finite. It can be viewed canonically as an \'etale map of orbifolds using the orbifold structure of~\cite{ACV}, or as a ramified map of schemes. We can also assume that both the source and target have maps to the source and target of $\mathcal C'$ and $\mathcal C$ that extend the map identifying the general fibers. To conclude, we tropicalize this family of covers and obtain the modification
\[
\Gamma'_1 \to \Gamma_1
\]
as claimed. 
\end{proof}

\section{Non-superabundant functions}\label{sec: genus 0}

We prove that every balanced function $\Gamma \to \RR$ is H-modifiable, provided it contracts no cycles. In particular, this implies realizability in genus $0$. The section includes technical constructions that are needed in later sections. 

\subsection{Basic structures} We begin with preparatory results about realizable or $H$-modifiable balanced functions.

\begin{theorem}[Closedness of realizability]\label{lemma:limit}
Let $(\Gamma_t, F_t)$ be a continuously varying family of tropical curves equipped with balanced functions, for $t \in [0,1)$. If $(\Gamma_t, F_t)$ is realizable for all $t$, then the limit $(\Gamma_1, F_1)$ is also realizable.
\end{theorem}

\begin{proof}
    Follows immediately from~\cite[Theorem A]{R16}.
\end{proof}

The theorem above will be used to reduce realizability statements to nearly trivalent cases. 

\begin{definition}
Let $F\colon \Gamma \to \Delta$ be a harmonic map of tropical curves. An edge or ray $e$ of $\Gamma$ is \textit{contracted} if $F$ is constant on $e$. A vertex in $\Gamma$ is a \textit{critical point} if it is simultaneously incident to contracted and non-contracted edges.
\end{definition}

\begin{lemma}\label{lemma:non-contracted_and_balanced_implies_realizable}
Let $F: \Gamma \to \RR$ be a balanced function on a possibly disconnected tropical curve such that no edge or ray of $\Gamma$ is contracted.
Then $F$ admits an H-modification $\hat{F}: \hat{\Gamma} \to \hat{\Delta}$.
\end{lemma}

\begin{proof}
Since no edge is contracted, the lemma reduces to checking local realizability at each vertex. Fix a vertex $v$ of the domain $\Gamma$. To $v$, we associate $\mathbb{P}^1$ marked at a finite collection of distinct points $S \subset \mathbb{P}^1$ in bijection with the tangent directions (flags) of edges based at $v$. At each point $p$ we have a slope $a_p$ coming from $F$. The divisor $\sum_{p \in S} a_p [p]$ has degree $0$ and is therefore principal. Any trivializing section of this line bundle gives a local realization of $F$ at $v$.



\end{proof}

%

The next result shows that modifications cannot make things ``worse'' for H-modifiability. Consequently, in many cases, we can assume that the harmonic maps we are interested in do not contract any rays.

\begin{lemma}\label{lemma:modification_of_realizable_is_realizable}
Let $F'\colon \Gamma' \to \mathbb{R}$ be a modification of a balanced function $F\colon \Gamma \to \mathbb{R}$. If $F$ is $H$-modifiable, then $F'$ is also $H$-modifiable.
\end{lemma}

\begin{proof}
It suffices to show that $F'$ is H-modifiable when a single additional ray is added to $\Gamma$. Denote this ray by $e$; its image under the collapsing map is a point $v$ in $\Gamma$. Hence, $F'$ maps $e$ to $F(v)$. 

By hypothesis, there exists an H-modification $\hat{F}\colon \hat{\Gamma} \to \hat{\Delta}$ of the original map $F\colon \Gamma \to \mathbb{R}$. We first glue the ray $e$ to $v \in \hat{\Gamma}$ in the canonical way; it has no target in the modification yet. Hence, we add a ray $r$ to $F(v) \in \hat{\Delta}$; this fixes the target of the putative H-modification of $F'$, while we only modify the domain later on.

In order to resolve local realizability at $v$, we further add $d_v(\hat{F})-1$ new rays to $v$; along with $e$, each is mapped to $r$ with slope $1$. Now, the only obstructions to local realizability are the vertices $u \in \hat{\Gamma}$ different from $v$ such that $\hat{F}(u)$ equals $F(v)$. Resolving a case like this is straightforward: we add $d_u(\hat{F})$ new rays to $u$ and map each to $r$ with slope $1$. We are done.
\end{proof}

\medskip

We recall that the \emph{star} of a vertex $v$ in a metric graph $\Gamma$ consists of $v$, all rays adjacent to $v$, and the interiors of all edges whose closures contain $v$. These stars may contain non-compact edges; the notions of harmonicity and local realizability extend verbatim to this setting.

\medskip

\begin{lemma}[Gluing lemma]\label{lemma:glueing}
Let $F\colon \Gamma \to \Delta$ be a harmonic map of (possibly disconnected) tropical curves. Suppose
\[
\Gamma=\Sigma\cup\bigcup_{i=1}^n C_i,
\]
where:
\begin{enumerate}[(i)]
        \item no edge or ray in $\Sigma$ is $F$-contracted,
    \item every edge and ray in each $C_i$ is $F$-contracted, and
    \item $C_i\cap C_j\subset\Sigma$ for $i\neq j$.
\end{enumerate}

Then $F|_{\Sigma}\colon\Sigma\to\Delta$ is harmonic. Assume it admits an $H$-modification.

For each $i$, define
\[
C_i'=\overline{\Bigl(C_i\cup\!\!\bigcup_{W\in\Sigma\cap C_i}\!\Star_{\Gamma}(W)\Bigr)\setminus\!\!\bigcup_{j\neq i}C_j}.
\]
The restriction $h_i \coloneqq F|_{C_i'}\colon C_i'\to F(C_i')$ is harmonic. Assume each $h_i$ admits an $H$-modification such that, in the modified domain, every edge or ray attached to $\Sigma\cap C_i$ has slope $1$.

Then $F$ admits an $H$-modification. On each $C_i$, the slopes on edges and rays agree with those prescribed by $h_i$.
\end{lemma}

\begin{proof}
We argue by induction on $n$; the case $n=0$ is trivial. Assume the statement holds for $n-1$ and set
\[
\Sigma_{n-1}=\Sigma\cup\bigcup_{i=1}^{n-1}C_i,
\qquad
F_{n-1}\coloneqq F|_{\Sigma_{n-1}}.
\]
Since every edge and ray in $C_n$ is contracted by $F$, the subgraph $C_n$ is mapped to a single vertex $V_n\in\Delta$. By the inductive hypothesis, $F_{n-1}$ admits an $H$-modification
\[
\widehat F_{n-1}\colon \widehat\Sigma_{n-1}\to\widehat\Delta_{n-1}
\]
with the required properties.

Define
\[
\Gamma_n=\widehat\Sigma_{n-1}\cup C_n
\]
and a map $G_n\colon\Gamma_n\to\widehat\Delta_{n-1}$ by
\[
G_n(x)=
\begin{cases}
\widehat F_{n-1}(x), & x\in\widehat\Sigma_{n-1},\\
V_n, & x\in C_n.
\end{cases}
\]
This is well-defined: for $i<n$ we have $C_i\cap C_n\subset\Sigma$, hence
\(
\widehat\Sigma_{n-1}\cap C_n=\Sigma\cap C_n,
\)
and on this intersection both definitions agree with $F$. The map $G_n$ is clearly harmonic.

By assumption, $h_n$ admits an $H$-modification $\widehat h_n$. Let $T_n=\operatorname{Im}(\widehat h_n)$, a tree rooted at $V_n$ arising from the modification, and let $\widehat C_n$ denote $C_n$ with all attached modifications. We now construct the desired modification.

\begin{itemize}
\item \textbf{Target.}
Obtain $\widehat\Delta$ from $\widehat\Delta_{n-1}$ by attaching a copy of $T_n$, identifying its root with $V_n$.

\item \textbf{Domain.}
Start with $\widehat\Sigma_{n-1}\cup\widehat C_n$, glued along $\Sigma\cap C_n$.  
For each vertex $W\in\widehat\Sigma_{n-1}\setminus(\Sigma\cap C_n)$ with $G_n(W)=V_n$, attach $d_W(\widehat F_{n-1})$ copies of $T_n$ at $W$, identifying their roots with $W$.  
Let $\widehat\Sigma$ be the resulting graph.

\item \textbf{Map.}
Define
\[
\widehat F\colon\widehat\Sigma\to\widehat\Delta
\]
to equal $\widehat F_{n-1}$ on $\widehat\Sigma_{n-1}$, to equal $\widehat h_n$ on $\widehat C_n$, and to map each newly attached copy of $T_n$ identically onto the target tree with slope $1$. 
\end{itemize}

The map $\widehat F$ is well-defined. For every vertex $W$ with $G_n(W)=V_n$, each edge or ray attached to $W$ outside $\widehat\Sigma_{n-1}$ has slope $1$; hence harmonicity and local realizability hold at every vertex. By construction, the natural diagram relating $F$ and $\widehat F$ commutes: newly added rays in the domain map to the corresponding components in the target.

Thus $\widehat F$ is an $H$-modification of $F$ with the desired properties.
\end{proof}
\color{black}

\subsection{Realizability without contracted cycles} 
We now work towards realizability for balanced functions with tree domains, and more generally, balanced functions that do not contract any cycles.

The following proposition serves as technical input to glue $H$-modifications of graphs along contracted trees. 

\begin{proposition}\label{theorem:trees}
Let $T$ be a genus $0$ tropical curve, and let
\[
F \colon T \to \mathbb{R}
\]
be a balanced function. Suppose there exists a compact subtree $T_1 \subset T$ such that $F$ contracts $T_1$ to $0 \in \mathbb{R}$. Assume furthermore that every ray of $T$ not contained in $T_1$ is non-contracted by $F$.

Let $A_1, \ldots, A_n$ ($n \ge 1$) be the points of $T_1$ that are incident to non-contracted rays of $T$. Assume that $T_1$, together with the stars of the points $A_i$, covers $T$, and that, aside from the edges belonging to $T_1$, only rays are attached at each $A_i$.

Then $F$ admits an $H$-modification
\[
\widehat{F} \colon \widehat{T} \to \widehat{\Delta}.
\]
Moreover, the modification may be chosen so that the slope of $\widehat{F}$ is equal to $1$ on every edge or ray contained in
\[
T_1 \,\cup\, \overline{\widehat{T} \setminus T}.
\]
\end{proposition}

\begin{proof}
We induct on the number of edges in the contracted subtree $T_1$. The base case is when $T_1$ consists of a single vertex, possibly with finitely many rays attached; the claim holds clearly. For the inductive step, we consider several different cases. In the course of the proof, we will call the vertices $A_1,\ldots,A_n$ the {\it critical vertices}. 

\begin{enumerate}[(i)]
\item Suppose there exists a critical vertex of valency at least $2$ in the tree $T_1$, and relabel so that this vertex is $A_1$. Since $T_1$ is a tree, we can disconnect $T_1$ at $A_1$, resulting in two strictly smaller subtrees, which we will call $T_1'$ and $T_1''$, such that $T_1' \cap T_1'' = \{A_1\}$.  
Let $S'$ be the union of $T_1'$ with all the stars in $T$ of the critical points that lie in $T_1'$, excluding the interior of $T_1''$.  
Let $S''$ be obtained analogously from $T_1''$.  
Since $F$ contracts any edges in $T_1'$ and $T_1''$, $F$ restricts to harmonic maps on $S'$ and $S''$, with targets being the respective images.  
Furthermore, these restrictions both admit $H$-modifications with the required properties, by the inductive hypothesis.  
Hence, the hypotheses of Lemma~\ref{lemma:glueing} is satisfied, so we get the required $H$-modification.

\item We now consider the case where every critical vertex $A_i$ is a leaf of $T_1$, and let $e_i$ denote the corresponding leaf edges. If any $e_i$ is incident to two critical vertices, then by connectivity of $T_1$, there must only be two critical vertices $A_1$ and $A_2$, and the edge $e_1$ between them is all of $T_1$.  
In this case, we modify as follows. Let $M$ be the midpoint of $e_1$ and let $a = |A_1 M|$, the length of $A_1 M$.  
For the domain modification, add two rays $\ell_1, \ell_2$ to the domain graph based at $M$.  
For the target, add a new ray $\ell$ to the target graph at the image of $M$, i.e., $0 \in \mathbb{R}$. Place a vertex $N$ on this new ray $\ell$, such that its distance to $0$ is $a$.  
Map the segments from $A_i$ to $M$, for both $i$, to the segment from $0$ to $N$. Map $\ell_1, \ell_2$ to the ray emanating from $N$ along $\ell$ with slope $1$.  
Finally, at each critical vertex $A_i$, add $d_{A_i} - 1$ rays to $A_i$, and map each to $\ell$ with slope $1$.

\item Assume that we are in neither of the previous situations, and that the critical vertices $A_i$ are all leaves with associated leaf edges $e_i$. In particular, none of the $e_i$ connect distinct $A_j$’s. Without loss of generality, let $e_1$ have minimal length among these leaf edges, and denote this length by $a$.

\textbf{Construction of the $H$-modification:}  
For each $i$, let $B_i$ be the point on $e_i$ at distance $a$ from $A_i$, and let $\mathcal{B}$ be the set of these points. For each distinct $P \in \mathcal{B}$, attach $d_P$ rays $\ell_{P,1},\dots,\ell_{P,d_P}$ to the domain at $P$. This produces a modification $T'$ of the domain $T$.  

Add a ray $\ell$ to the target graph at $0 \in \mathbb{R}$, and place a vertex $N$ on $\ell$ at distance $a$ from $0$. This yields a modification $\Sigma'$ of the target. Map each newly added ray in the domain to $\ell$ with slope $1$, starting at $N$.

\textbf{Applying the inductive hypothesis:}  
Remove the half-open segments $[A_i,B_i)$ from $T$, giving a disjoint union of trees $S_1,\dots,S_m$, each containing strictly fewer edges than $T_1$. For each $j$, form $S'_j$ by attaching the stars of all vertices in $\mathcal{B} \cap S_j$ to $S_j$.  

Viewing the interior of $\ell$ as a copy of $\mathbb{R}$, each $S'_j$ induces a map $h_j \colon S_j \to \mathbb{R}$ that contracts $S_j$ to $N$. By the inductive hypothesis, each $h_j$ admits an $H$-modification satisfying the required properties.  
Applying the gluing construction of Lemma~\ref{lemma:glueing}, we obtain an $H$-modification of the original map $T \to \mathbb{R}$ with the desired properties.
\end{enumerate}

The three situations exhaust the inductive step, so the result follows by induction on the number of vertices of the contracted subtree. 
\end{proof}

We now have the ingredients to show that maps without contracted cycles are $H$-modifiable. 

\begin{theorem}\label{theorem:no contracted cycles}
Suppose $F\colon \Gamma \to \mathbb{R}$ is a balanced function on a tropical curve that does not contract any cycles. Then $F$ is realizable.
\end{theorem}

\begin{proof}
Any $F$-contracted rays in $\Gamma$ can be regarded as arising from a modification, so by Lemma~\ref{lemma:modification_of_realizable_is_realizable} we may assume that $\Gamma$ has no such rays. Let $\Sigma$ be the largest metric subgraph of $\Gamma$ without contracted edges, and let $C$ be the closure of its complement. By hypothesis, $C$ is a union of disjoint compact metric trees $T_1, \dots, T_n$, and if $C$ is nonempty, then $\Sigma \cap C$ is finite. By Lemma~\ref{lemma:non-contracted_and_balanced_implies_realizable}, the restriction $F|_{\Sigma}$ admits an $H$-harmonic modification $F|_{\Sigma}'$ that agrees with $F$ on $\Sigma$.

For each $T_i$, let $T_i'$ be the union of $T_i$ with the stars of those points in $T_i$ that meet $\Sigma$. By Proposition~\ref{theorem:trees}, the restriction of $F$ to $T_i'$ admits an $H$-modification with slope $1$ on all edges or rays that were contracted by $F$ or newly added. Applying Lemma~\ref{lemma:glueing} to glue these modifications to $F|_{\Sigma}'$ produces an $H$-modification of $F$, as required.
\end{proof}

In particular, this implies the following:

\begin{corollary}\label{lemma:genus_0_implies_realizable}
Let $\Gamma$ be a tropical curve of genus $0$ and let $F\colon\Gamma\to\RR$ be a balanced function. Then $F$ is $H$-modifiable, and therefore realizable.
\end{corollary}



\section{Genus one results}\label{sec: genus 1}
We consider balanced functions on tropical curves of genus $1$, focusing on the case where the cycle is contracted. Our goal is to \emph{discover} well-spacedness by studying modifiability.

\subsection{The well-spacedness condition} 
We recall the statement of well-spacedness. Following~\cite{RSW17B}, the condition here differs slightly from Speyer's definition outside the trivalent case.

In this section, let $\Gamma$ denote a tropical curve of genus $1$. It has a unique cycle, and the vertices on this cycle are called \emph{cycle vertices}. Every point of $\Gamma$ lies on a unique path leading to the cycle. Suppose
\[
F\colon \Gamma \to T
\]
is a harmonic map that is constant on the cycle. A point $P \in \Gamma$ is a \emph{critical point} if $F$ is constant along the path from $P$ to the cycle but non-constant on some edge incident to $P$. Intuitively, a critical point is the first point along a path from the cycle at which $F$ begins to vary.

\begin{definition}
Suppose $F\colon \Gamma \to \mathbb{R}$ is balanced, and let $V$ be a critical point of $\Gamma$. The \emph{flag number} of $F$ at $V$, denoted $n(V)$, is the number of edges or rays incident to $V$ that are not contracted by $F$. We say $V$ is a \emph{simple} critical point if $n(V) = 2$.
\end{definition}

In what follows, the path from a critical point to the cycle of $\Gamma$ will be called a \emph{critical path}. 

\begin{definition}
    Suppose $F\colon \Gamma \to \RR$ is a parameterized tropical curve of genus $1$, and assume that $F$ is constant on an open neighborhood of the cycle. Then $F$ is \emph{well-spaced} if the sum of the flag numbers of the critical vertices on critical paths of minimal length is at least $3$.
\end{definition}

If $\Gamma$ is trivalent and $V$ is a critical point of a balanced map, then $n(V)$ must be $2$. Thus, when $\Gamma$ is trivalent, the condition says that there are at least two critical paths of minimal length, which is the condition stated in~\cite{Sp07}. The flag number will only be important outside the trivalent case. Our goal is the following restatement of Theorem~\ref{thm: well-spacedness}.

\begin{theorem}
Let $F\colon\Gamma\to\RR$ be a balanced function on a tropical curve of genus $1$. Assume that $F$ is constant on a neighborhood of the cycle. Then $F$ is $H$-modifiable, if and only if it is well-spaced. In particular, it is realizable, if and only if it is well-spaced. 
\end{theorem}

\subsection{Necessity of the well-spacedness condition}\label{sec: genus-one-necessity} 

We will often study directed paths, which we will denote by $w$, and will typically drop the adjective ``directed'' but specify the direction explicitly when needed. A path consisting of $k$ edges has finitely many vertices $P_1,\ldots,P_{k+1}$, and we will sometimes use the notation\footnote{Technically, a path need not be determined by its vertices when the graph has parallel edges. Parallel edges can always be removed by subdividing and adding $2$-valent vertices to the vertex set. This subtlety will not affect our discussion, but we may assume such a subdivision has been made.} 
\[
w = \overline{P_1,\ldots,P_{k+1}}.
\] 
Given such a path, we say $P_1$ is the \emph{starting vertex} of $w$, $P_{k+1}$ is the \emph{terminal vertex}, and any other vertices are the \emph{interior vertices}.

We now introduce two important notions required for the analysis of modifiability. 

\begin{definition}
Let $F\colon\Gamma\to \Sigma$ be a harmonic map. Let $w$ be a path in $\Sigma$ consisting of the sequence of edges $e'_1,\dots,e'_k$. A path $v$ in $\Gamma$, having the sequence of edges $e_1,\dots,e_l$, is a \emph{lifting} of $w$ if $l = k$ and 
\[
F(e_i) = e'_i \quad \text{for each } i = 1,\dots,k.
\]
\end{definition}

\begin{definition}
Let $F\colon\Gamma\to \Sigma$ be a map, and let $w$ be a path in $\Gamma$ consisting of a sequence of edges $e_1,\dots,e_k$ joining vertices $P_1,\dots,P_{k+1}$. An interior vertex $P_i$ of $w$ is called a \emph{turning point} if $F$ is not contracted on $e_{i-1}$ and there exists an edge $e \in \Star_\Gamma(P_i)$, distinct from $e_{i-1}$, such that 
\[
F(e) = F(e_{i-1}).
\]
\end{definition}

The next lemma uses these notions to show that harmonicity allows us to ``copy'' paths, eventually producing the multiple critical paths that appear in the well-spacedness condition. 

\begin{lemma}\label{lemma:copying_paths}
Let $F\colon \Gamma\to \Sigma$ be harmonic. Let $w$ be a path from a vertex $A$ to a vertex $B$ such that $F$ is injective on $w$. Then, for any $B' \in \Gamma$ with $F(B') = F(B)$, there exists a path $v$ in $\Gamma$ lifting $F(w)$ that terminates at $B'$. 

Moreover, if $B$ is a turning point and $B' = B$, it is possible to choose $v \neq w$.
\end{lemma}

\begin{proof}
Let $w$ be a path from $A$ to $B$ consisting of edges $e_1,\dots,e_k$. Then $F(w)$ is a path in $\Sigma$. By injectivity, the edges $F(e_1),\dots,F(e_k)$ are distinct, with vertices $F(A),P_2,\dots,P_k,F(B)$.  

Since $F(B') = F(B)$, there exists an edge $e'_k \in \Star_{\Gamma}(B)$ with 
\[
F(e'_k) = F(e_k),
\] 
which follows from the local surjectivity of harmonic maps. If $B' = B$ and $B$ is a turning point, we can ensure that $e'_k \neq e_k$. This guarantees the second statement of the lemma.  

Let $F(e'_k)$ have vertices $F(B)$ and $P'_k$. Since $F(e_{k-1})$ is incident to $P'_k$, harmonicity of $F$ ensures the existence of an edge $e'_{k-1}$ with vertices $P'_k$ and $P'_{k-1}$ satisfying 
\[
F(e'_{k-1}) = F(e_{k-1}).
\] 
Inductively, this produces edges $e'_1,\dots,e'_k$ with $B$ a vertex of $e'_k$ and $F(e_i) = F(e'_i)$. Then the path $v$ in $\Gamma$ consisting of this sequence of edges is a lifting of $w$ that terminates at $B'$.
\end{proof}

We now record a technical consequence of the lemma. 

\begin{proposition}\label{prop:copying_paths}
Let $F\colon \Gamma \to \RR$ be a balanced function, and let $\hat{F}\colon \hat{\Gamma} \to \hat{\Sigma}$ be an $H$-harmonic modification of $F$. Let $A$ be a simple critical vertex not on the cycle, let $B$ be a turning point with respect to $\hat F$, and let $w$ be an $F$-constant path of edges with no interior $\hat F$-turning vertices or $F$-critical vertices. Then the following hold:
\begin{enumerate}[(i)]
    \item The map $\hat F$ is injective on $w$ with slope $1$. In particular, $\hat F(w)$ is a path in $\hat{\Sigma}$.
    \item For any point $B'$ in $\Gamma$ with $\hat F(B') = \hat F(B)$, there exists an $F$-critical vertex $A'$ and a path $v$ from $A'$ to $B'$ that lifts $\hat F(w)$, and the length of $v$ is at most that of $w$.
    \item When $B = B'$, we can ensure that $A' \neq A$ in the construction above.
\end{enumerate}
\end{proposition}

\begin{proof}
Let $w$ be the path described in the statement, with vertices $A, P_1, \dots, P_{k-1}, B$ and edges $e_1, \dots, e_k$. We first examine the local structure at $A$. Since $A$ is a critical point, the balancing condition implies there are at least two edges or rays, $\ell_1$ and $\ell_2$, incident to $A$ that are not contracted by $F$. Because $A$ is a \emph{simple} critical vertex, there are exactly two such edges, and the slopes of $F$ along $\ell_1$ and $\ell_2$ coincide; denote this slope by $s$. The local degree of $\hat F$ at $A$ is the same as that of $F$, which we also denote by $s$.

Next, we consider the slope of $\hat F$ along $e_1$ and show that it must be $1$. Analyzing the corresponding local map of algebraic curves at $A$, we have a degree-$s$ map with two fully ramified points corresponding to $\ell_1$ and $\ell_2$. By the Riemann--Hurwitz formula, no further ramification is possible, so the slope of $e_1$ under $\hat F$ is necessarily $1$.

We now turn to the interior vertices and edges of $w$ and claim that their local degrees under $\hat F$ are also $1$. Suppose, for contradiction, that $P_i$ is the first vertex where the local degree is not $1$. The edge $e_i$ has slope $1$ by assumption, so there must be some other edge attached to $P_i$ with the same image under $\hat F$. This would make $P_i$ a turning vertex, contradicting our assumptions. Hence, $\hat F$ has slope $1$ along the entire path $w$, proving the first claim.

For the second claim, Lemma~\ref{lemma:copying_paths} guarantees a path $v$ in $\hat\Gamma$ starting at a vertex $A'$ and ending at $B'$ that lifts $\hat F(w)$. Write $v = \overline{A', Q_1, \dots, Q_{k-1}, B'}$. We claim that $A'$ lies in the original $\Gamma$. Indeed, since $A$ is a critical point, $F(A)$ lies in the original target $\mathbb{R} \subset \hat\Sigma$, so $\hat F(A') \in \mathbb{R}$ and $A'$ must lie in $\Gamma$. If any of the vertices $Q_i$ left $\Gamma$, the path could not return to $B'$, so they must also lie in $\Gamma$.

Since $w$ has no interior critical vertices, for all interior vertices $P_i$ and $Q_i$, we have $\hat F(P_i) = \hat F(Q_i)$ lying outside the original copy of $\mathbb{R}$ in $\hat\Sigma$. Thus, the edge $A'Q_1$ is contracted by $F$, and since $\hat F(A') \in \mathbb{R}$, $A'$ is an $F$-critical vertex.

For the third claim, if $B = B'$, then $A \neq A'$, since otherwise $w$ and $v$ would share both endpoints, forming a cycle in $\Gamma$ and contradicting the assumption that $A$ is not a cycle vertex.

Finally, regarding lengths, note that $w$ and $v$ have the same image under $\hat F$. The slope of $\hat F$ on $w$ is $1$, while on edges of $v$ it may be any integer $\ge 1$. Therefore, the length of $w$ is at least that of $v$.
\end{proof}

We now prove necessity of the well-spacedness condition. 

\begin{theorem}[Necessity]\label{proposition: speyer contracted 1}
Let $F\colon \Gamma\to\RR$ be a realizable balanced function on a tropical curve of genus $1$. Assume an open neighborhood of the cycle is contracted. Then $F$ is well-spaced. 
\end{theorem}

\begin{proof}
We assume the existence of an $H$-modification $\hat{F}\colon\hat{\Gamma} \to \hat{\Sigma}$ and show that there cannot be a simple critical point whose associated critical path has the unique minimal length among all critical paths. We argue by contradiction and, to set it up, suppose such a path exists and is given by $w$ with vertices ${P_1,\dots,P_n}$, with $P_1$ the only simple critical point on this path, and $P_n$ is the only point on the cycle. 

We first claim that none of $P_2,\dots,P_n$ is a turning point. If $P_i$ is a turning point, Proposition~\ref{prop:copying_paths} gives us a new path $\overline{Q_1,Q_2,...,Q_i}$ with $F(Q_k) = F(P_k)$ and such that $Q_1$ is a critical point. By the length considerations in this proposition, we contradict the length-minimumity or simplicity hypotheses.

We now examine the algebraic map at the point $P_n$. We claim that the local degree is $1$. By applying again Proposition~\ref{prop:copying_paths}, the slope of $\hat{f}$ for the edge $P_{n-1}P_n$ is $1$. Suppose that the local degree at $P_n$ is greater than $1$, a different edge in $\hat{\Gamma}$ attached to $P_n$ has the same image as $P_{n-1}P_n$. This contradicts the turning point hypothesis. 

Now we can conclude the argument. Consider the subgraph consisting of the following three edges: $P_{n-1}P_n$, $P_n A$, and $P_n B$, where $A$ and $B$ are the two vertices on the cycle incident to $P_n$. Since the local degree at $P_n$ is $1$, we know the images of $P_nA$ and $P_nB$ differ under $\hat F$. Now go around the cycle from $P_n$. By continuity of the map $\hat F$, there must be a point on the cycle distinct from $P_n$, whose image has the same image under $\hat F$ as $P_n$. By Proposition~\ref{prop:copying_paths}, there is a path $\overline{R_1,\dots,R_n}$ such that $\hat{F}(R_i) = \hat{F}(P_i)$ for all $i$, with $R_1$ is a critical point, and the length of this path is at most that of $w$. This gives the contradiction.
    
\end{proof}

\subsection{Preparation disguised as illustrative examples}

We present several examples illustrating realizability and well-spacedness via modifications. These examples are intended to aid the reader’s intuition, but they also serve as preparatory base cases for the proof of sufficiency of well-spacedness in the next section.

\subsubsection{A single ray attached to the cycle}

As a first example, we explain how realizability may be obstructed. We begin with the following lemma.

\begin{lemma}
Let $\hat{F}\colon\hat{\Gamma}\to\hat{T}$ be a modification of $F\colon\Gamma\to\mathbb{R}$ with no contracted edges. Then there exists a point $P$ on the contracted subgraph of $\Gamma$ that is a turning point with respect to some path from a critical point $A$ to $P$.
\end{lemma}

\begin{proof}
A walk can be constructed on the contracted subgraph starting and ending at $A$. The modification $\hat{F}$ has the same value for the starting and ending points of the walk. By the mean value theorem, there is a turning point on this subgraph.
\end{proof}

We can now give the example in the form of the following:

\begin{proposition}
The map $F\colon\Gamma\to\mathbb{R}$ in Figure~\ref{fig: g1_speyer_one_leg} is not $H$-modifiable, and therefore not realizable.
\end{proposition}

\begin{figure}[h!]
\includegraphics{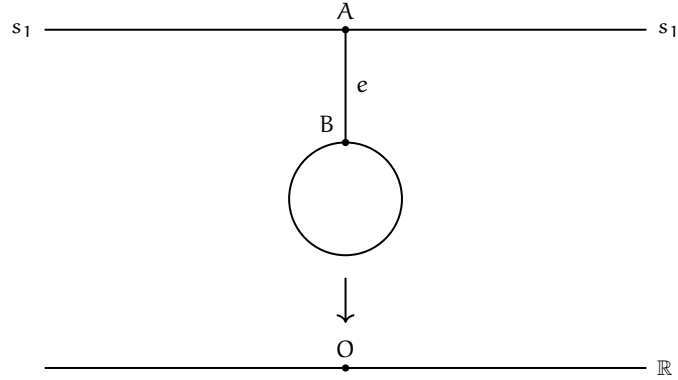}
\caption{A tropical map with a contracted cycle and a single non-well-spaceable critical path. To interpret the figure, note that the edge $AB$ and the cycle are both contracted. The map is linear of slope $s_1$ on the rays based at $A$, in the indicated direction.}\label{fig: g1_speyer_one_leg}
\end{figure}

\begin{proof}
Suppose $\hat{F}\colon\hat{\Gamma}\to\hat{T}$ is an H-harmonic modification of $F\colon\Gamma\to\mathbb{R}$. Then by the previous lemma, there is a turning point $P$ with respect to $\hat{F}$ and some path which lies on the contracted subgraph in $\Gamma$. Assume without loss of generality that $P$ is the closest turning point to $A$ on that path. By Proposition~\ref{prop:copying_paths}, there is another critical point distinct from $A$, which is a contradiction. 
\end{proof}

If instead the critical point had three outgoing non-constant directions, realizability holds. 

\begin{proposition}\label{lem: genus 1 with 4-valent-critical-point}
The map $F\colon\Gamma\to\RR$ in Figure~\ref{fig: g1_speyer_one_leg_threeflags} is $H$-modifiable and therefore realizable. 
\end{proposition}

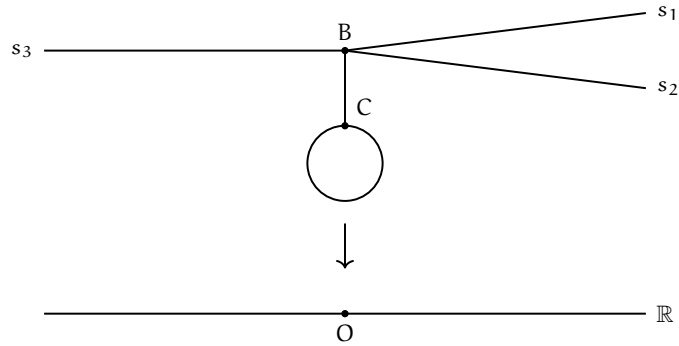
\begin{figure}[h!]
\begin{tikzpicture}
\draw (0,3) -- (0,2);
\draw (0,1.5) circle (0.5);

\filldraw[black] (0,3) circle (1pt) node[anchor=south,font=\footnotesize] {$B$};
\filldraw[black] (0,2) circle (1pt) node[anchor=south west,font=\footnotesize] {$C$};

\draw (0,3) -- (4,3.5) node[anchor=west,font=\scriptsize] {$s_1$};
\draw (0,3) -- (4,2.5) node[anchor=west,font=\scriptsize] {$s_2$};
\draw (0,3) -- (-4,3) node[anchor=east,font=\scriptsize] {$s_3$};

\draw[->] (0,0.7) -- (0,0.1);

\draw (-4,-0.5) -- (4,-0.5) node[anchor=west,font=\footnotesize] {$\mathbb{R}$};
\filldraw[black] (0,-0.5) circle (1pt) node[anchor=north,font=\footnotesize] {$O$};
\end{tikzpicture}
\caption{A tropical map with a contracted cycle and a single critical path with three non-constant flags at the critical point. In the figure, the cycle and the edge $BC$ are contracted. The map is linear of slope $s_i$ on the three rays, in the indicated directions.}\label{fig: g1_speyer_one_leg_threeflags}
\end{figure}

\begin{proof}
We describe the modification $\hat F$. First, modify $\RR$ by adding a ray at $O$. Map the segment $BC$ onto this new ray with slope $2$. Next, modify the circle by adding a ray at the antipodal point of $C$, and denote the new vertex by $C'$. Map both paths from $C$ to $C'$ onto the new ray with slope $1$, and then map the remainder of the new ray in the domain onto the remainder of the new ray in the target with slope $2$.

All required Hurwitz covers can be realized. The only nontrivial verification occurs at the vertex $B$ in the modification. A Riemann--Hurwitz computation shows that any cover realizing the prescribed local map at $B$ contains exactly one simple ramification point, as required by the slope-$2$ condition along $BC$.

A figure illustrating an analogous construction appears in the proof of Proposition~\ref{construction: genus 1 with one Y-leg_type A}.
\end{proof}

\begin{remark}[A purely geometric argument]
In the above situation, since the map is well-spaced regardless of the length of the critical path, one can describe the construction in more classical geometric language. For simplicity of exposition, consider the case where the circle in Figure~\ref{fig: g1_speyer_one_leg_threeflags} is shrunk to a single genus-$1$ vertex. Choose an elliptic curve $E$ with a marked point $p$ corresponding to the edge $BC$. Choose a $2\!:\!1$ map $E \to \PP^1$ ramified at $p$. The curve $C_B$ at $B$ and its map can be realized by a rational function on $\PP^1$, and a Riemann--Hurwitz computation shows that there is a unique additional simple ramification point $q$, beyond the marked points. Attach $E$ to this $\PP^1$ by identifying $p$ with $q$. The target is glued analogously, yielding the desired admissible cover.
\end{remark}

\subsubsection{Two paths from the cycle}

We next consider a slightly more complicated situation: a superabundant tropical curve that is well-spaced due to the relative positions of two critical paths.

\begin{proposition}\label{prop:g1_speyer_two_legs}
Consider the map $F\colon \Gamma\to\RR$ shown in Figure~\ref{fig: g1_speyer_two_legs}. If the lengths of the edges $e_1$ and $e_2$ are equal, then $F$ is $H$-modifiable.
\end{proposition}

\begin{figure}[h!]
\includegraphics{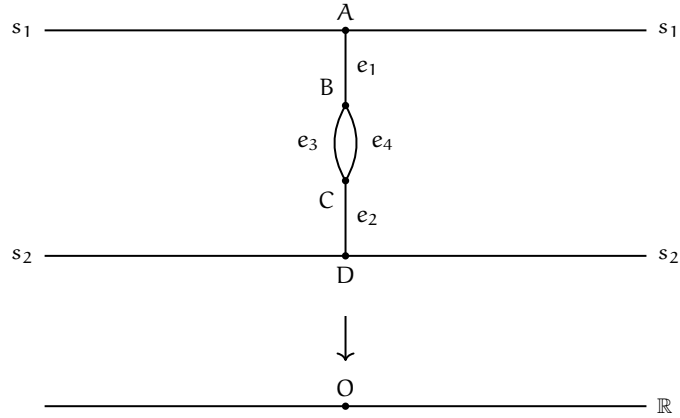}
\caption{A tropical map with a contracted cycle with two critical paths. To interpret the figure, note that the cycle, $AB$, and $CD$ are all contracted. The map is linear of slope $s_1$ resp. $s_2$ on the upper resp. lower rays.}\label{fig: g1_speyer_two_legs}
\end{figure}

\begin{proof}
If $e_1$ and $e_2$ have the same length, the construction in Figure~\ref{fig: basic-Speyer} produces an $H$-modification
\begin{figure}
\includegraphics[scale=0.75]{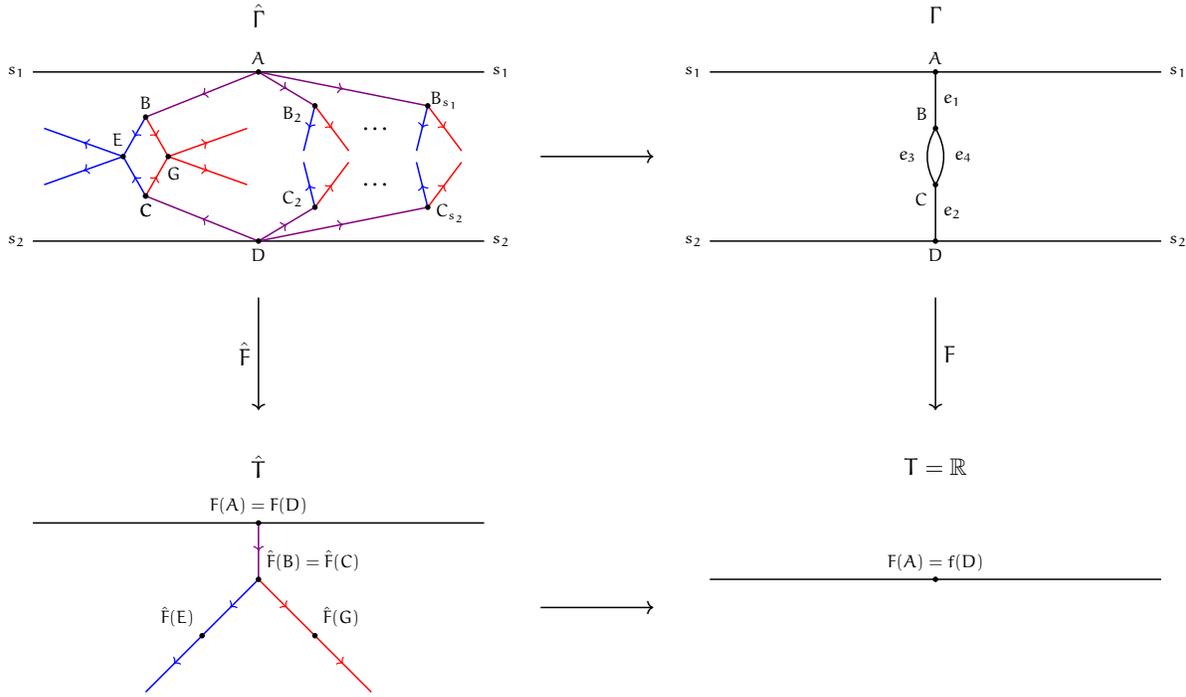}
\caption{The modification required for the proof of Proposition~\ref{prop:g1_speyer_two_legs}. The contracted parts of the original graph are shown in colour in $\Gamma'$, as are the required new rays in the modification. The ray in $\hat T$ that points southeast receives the image of $e_3$ in the original $\Gamma$, and the southwest receives $e_4$. The remainder of the modification accounts for the possibility that the two $s_i$ are larger than $1$.}\label{fig: basic-Speyer}\phantom{.} 
\end{figure}
The map $\hat{F}$ is chosen so the slope on the colored parts -- all edges that are either newly added or contracted by $F$ -- have slope $1$. The modification is harmonic and locally realizable: at $A$ resp. $D$ consider the maps $z\mapsto z^{s_1}$ resp. $z\mapsto z^{s_2}$. Use the identity map at all other vertices. 
\end{proof}

Next, we consider an example where there are two critical paths with a shared edge.

\begin{proposition}\label{construction: genus 1 with one Y-leg_type A}
Let $F\colon\Gamma\to\RR$ be the map in Figure~\ref{fig: g1_speyer_y_shaped_leg}. If the distances from $A_1$ and $A_2$ to $B$ are equal, then $F$ is $H$-modifiable.
\end{proposition}

\begin{figure}[h!]
\includegraphics{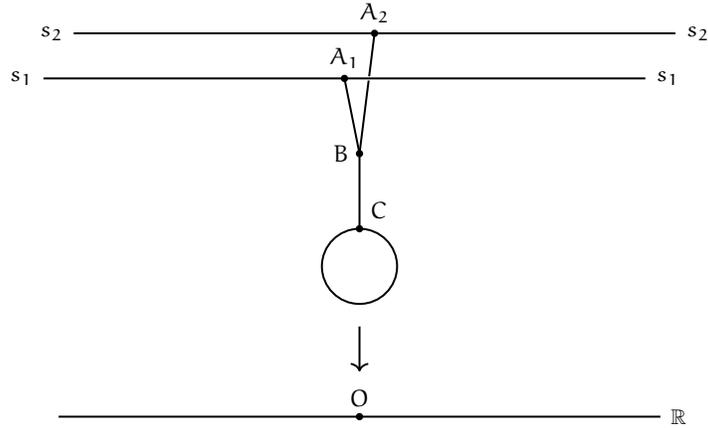}
\caption{A tropical map with a contracted cycle with two critical paths forming a ``Y'' to the cycle. To interpret the figure the figure, note that the cycle, $BC$, $A_1B$, and $A_2B$ are all contracted.}\label{fig: g1_speyer_y_shaped_leg}
\end{figure}

\begin{proof}
The proof is similar to that of Proposition~\ref{lem: genus 1 with 4-valent-critical-point}. If $A_1$ and $A_2$ have the same distance from $B$, then modify $\mathbb{R}$ by adding a ray to $O$, and another two rays at the points (on the added ray) with distance $|A_1 B|$ and $|A_1 B|+2|BC|$ from $O$ respectively. Make the same modifications at $A_1,A_2\in\Gamma$ $s_1-1$ and $s_2-1$ times respectively, and naturally map these onto $\hat{T}$ with slope $1$. Note that in the figure below, we suppress these extra rays. Next, modify $\Gamma$ by adding one ray at $B$, one ray at $C$, and two rays at the midpoint of the loop at $C$. Then let $\hat{F}:\hat{\Gamma}\to\hat{T}$ be as follows in Figure~\ref{fig: Speyer-Y-shaped} below. Then $\hat{F}$ is an $H$-modification of $F\colon\Gamma\to\RR$.

\begin{figure}[h!]
\includegraphics[scale=0.75]{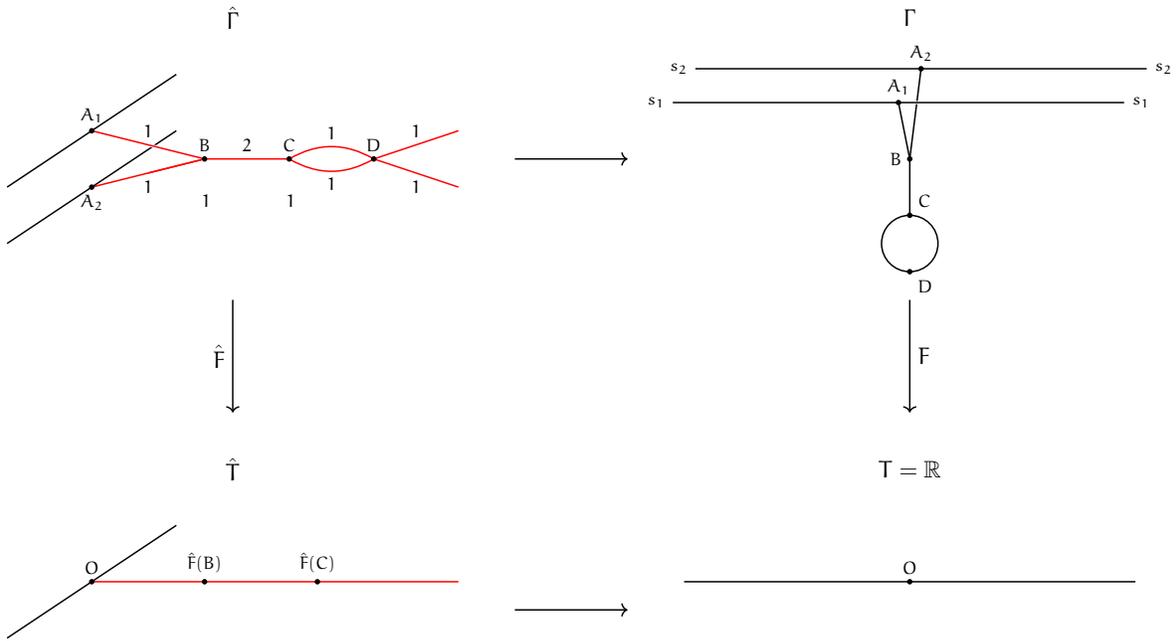}
\caption{The modification construction required for Proposition~\ref{construction: genus 1 with one Y-leg_type A}. The subgraph that was formerly the edges $A_1B$, $A_2B$, $BC$, and the cycle are now in red, as are the new rays in the modification. The numbers on $\hat\Gamma$ are slopes of the edges as they map onto the indicated edges. The interior edges in the modified part of the graph is trivalent, and since the map to the modified ray on the target is balanced, the Hurwitz problem is solvable on all these edges.}\label{fig: Speyer-Y-shaped} 
\end{figure}
\end{proof}

\subsection{The full realizability criterion} With the cases above done by hand, we now explain how to inductively deduce the general case. 

\begin{theorem}[Sufficiency]\label{proposition: speyer contracted 2}
Let $F\colon \Gamma\to\RR$ be a balanced function on a genus-$1$ tropical curve. Assume that an open neighborhood of the cycle is contracted. If $F$ is well-spaced, then it admits an $H$-modification and is therefore realizable. 
\end{theorem}

We begin with a basic reduction.

\begin{lemma}\label{lem: generic-lengths}
The general case of Theorem~\ref{proposition: speyer contracted 2} follows from the special case in which:
\begin{enumerate}[(i)]
\item There are at most two critical paths of minimal length.
\item The graph $\Gamma$ is trivalent if there are two minimal critical paths, and otherwise is trivalent away from a single $4$-valent vertex located on the (unique) minimal critical path.
\item The critical paths of non-minimal length have pairwise distinct lengths.
\end{enumerate}
\end{lemma}

\begin{proof}
This follows immediately from the closedness of the realizable locus in the moduli cone of maps containing $[F]$, as stated in Theorem~\ref{lemma:limit}. 
\end{proof}

In what follows, we assume the conditions of the lemma hold. We denote the minimal critical path by $A_1B_1$, where $B_1$ lies on the cycle and $A_1$ is the critical point.

The structure of our proof of the theorem above is to induct on the number of non-minimal critical trees, with the base cases being the ones above. In the induction, a few finer properties of the modification are relevant. Before proceeding further, we record a combinatorial trick that arises frequently in practice, and is relevant for the proofs and definitions that follow. 

\begin{remark}[Long images and slope adjustments]\label{rem: slope-adjustment}
Consider $F\colon \Gamma\to\RR$ as in Figure~\ref{fig: g1_speyer_y_shaped_leg}. If we add a new critical path starting at a new point of the cycle, we have the following: 

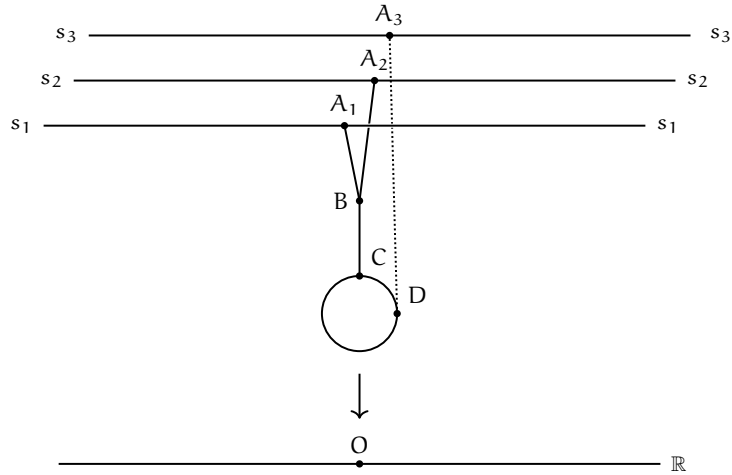
\begin{figure}[h!]
    \begin{tikzpicture}
        \draw[densely dotted] (0.5,1.5)--(0.4,5.2);        
        \draw (0,3) -- (0,2);
        \draw (0,1.5) circle (0.5);
        \draw (0,3) -- (-0.2,4);
        \draw (0,3) -- (0.1,4.6);
        \draw (0.4,5.2)  --(4.4,5.2);
        \draw (0.4,5.2)--(-3.6,5.2);
        \draw node[anchor=east,font=\scriptsize] at (5.1,5.2) {$s_3$};
        \draw node[anchor=east,font=\scriptsize] at (-3.6,5.2) {$s_3$};
        
        \draw (-3.9,4.6) node[anchor=east,font=\scriptsize] {$s_2$}  -- (4.1,4.6) node[anchor=west,font=\scriptsize] {$s_2$};
        \draw[white,ultra thick] (-4.2,4) -- (3.8,4);
        \draw (-4.2,4) node[anchor=east,font=\scriptsize] {$s_1$}  -- (3.8,4) node[anchor=west,font=\scriptsize] {$s_1$};
        
        \filldraw[black] (-0.2,4) circle (1pt) node[anchor=south ,font=\footnotesize] {$A_1$};
        \filldraw[black] (0.1,4.6) circle (1pt) node[anchor=south ,font=\footnotesize] {$A_2$};
        \filldraw[black] (0.4,5.2) circle (1pt) node[anchor=south ,font=\footnotesize] {$A_3$};
        \filldraw[black] (0,3) circle (1pt) node[anchor=east,font=\footnotesize] {$B$};
        \filldraw[black] (0,2) circle (1pt) node[anchor=south west,font=\footnotesize] {$C$};
        \filldraw[black] (0.5,1.5) circle (1pt) node[anchor=south west,font=\footnotesize] {$D$};
        
        \draw[->] (0,0.7) -- (0,0.1);
        
        \draw (-4,-0.5) -- (4,-0.5) node[anchor=west,font=\scriptsize] {$\mathbb{R}$};
        \filldraw[black] (0,-0.5) circle (1pt) node[anchor=south,font=\footnotesize] {$O$};
    \end{tikzpicture}
    \caption{A new critical path $B$ attached to the cycle. The cycle and the segments $BC$, $A_iB$, and $A_3D$ are all contracted.}\label{fig: g1_speyer_y_shaped_leg_appended}
\end{figure}

Before this new attachment, the standard $H$-modification adds a ray based at $O$ in the target; it then maps $A_1B$ and $A_2B$ onto this ray with slope $1$. Balancing forces the outgoing slope at $B$ to be $2$, and we map $BC$ with slope $2$, the two edges of the loop with slope $1$, and then modify by adding two rays at a point of the loop. 

We now contemplate adding $A_3D$. It can happen that $|A_3D|$ is bigger than $|BC|+|A_1B|$, but smaller than $|A_1B|+2|BC|$. Now $A_3$ and $A_1$ map to the same point on the target, and since the slope of the modification $A_3D$ may be forced to be $1$, this can make it impossible to extend the standard modification of Figure~\ref{fig: g1_speyer_y_shaped_leg} to accommodate the new critical tree. 

\begin{figure}
    \centering
    \includegraphics[scale=0.75]{Figures/g1_speyer_y_shaped_leg_appended_modification.tex}
    \caption{An illustration of the slope adjustment trick. The edge $BC$ is divided at $E$; the segment $BE$ is mapped with slope $2$ while $EC$ is mapped with slope $1$. The point $E$ is chosen to ensure that the length of the image of the path from $A_1$ to $C$ matches $|A_3D|$. The tree attached at $E$ ensures harmonicity where the slope changes.}
    \label{fig:g1_speyer_y_shaped_leg_appended_modification}
\end{figure}
The solution is as follows. Choose a point $E$ on $BC$ such that $|A_1B|+2|BE|+|EC|=|A_3 D|$
length of $A_1B$, plus twice the length of $BE$, plus the length of $EC$ is equal to the length of $A_3D$
. Map $BE$ with slope $2$ as before, but map $EC$ with slope $1$. This creates an issue with harmonicity at $E$, but we will come back to fix this later. The image of $C$ is a point of the new ray added at $O$ in the target. Add another infinite ray at this image, so there are now two rays based at the image of $C$; call them $r$ and $r'$. Now, map the clockwise arc onto $r$ by dividing the arc at its midpoint, map the first half onto $r$, and map the second half back towards the image of $C$. The midpoint of this arc is a turning point. By adding further rays at this turning point, we can ensure harmonicity. Do the same for the counter-clockwise arc, mapping it to $r'$. 

Finally, to fix the issue at the point $E$, we attach a tree to $E$. Precisely, if we delete the image of $E$ from the target, there are two connected components. Let $T$ be the (closure of the) component containing the image of the cycle. Attach a copy of $T$ to the domain based at $E$ and map it down isomorphically (with slope $1$). This produces an $H$-modification.
\end{remark}

The remark above shows that the lengths of images of paths in modifications plays a role in the analysis, and also records the ``slope adjustment trick'', which we will use later. 

We are now ready to proceed with the rest of the proof. 

\begin{definition}
Let $F\colon \Gamma\to\RR$ be as above. An $H$-modification $\hat F$ is called {\it cycle decomposing} if it satisfies the following three conditions:
\begin{enumerate}[(i)]
\item For any minimal critical path $w$ in $\Gamma$ the image $\hat F(w)$ is still a path, and for any critical path $u$, the length of $\hat F(u)$ is larger than or equal to that of $\hat F(w)$. 
\item There exists $r>0$ such that, if we let $T_r = \hat F(A_1B_1)\cup \mathbb B_r(\hat F(B_1))$, where the latter is the ball of radius $r$ around $\hat F(B_1)$,  then 
\[
\Gamma\setminus \hat F^{-1}(T_r)
\]
is a disjoint union of trees, and the leaf vertices of each tree have the same image under $\hat F$. 
\item Any $F$-contracted edge that does not lie on the cycle nor on a minimal critical path has slope $1$ under $\hat F$.
\end{enumerate}
\end{definition}

The examples/base cases discussed above are each cycle decomposing. 

In the process of adding non-minimal critical paths, it is more straightforward to do so if the path is ``very long'' -- we make this precise below. Recall that $A_1B_1$ is always a minimal critical path in our notation. 

\begin{definition}
Let $F\colon \Gamma\to\RR$ be as above, and $\hat F$ a cycle decomposing $H$-modification. Add a new critical path $w$ with critical point $A$, and meeting $\Gamma$ at $B$. We say that $w$ is {\it very long} if the following conditions hold. 
\begin{enumerate}
\item If $B$ is in a critical path with critical point $A_2$, then the length of $\hat F(A_2B)$ is shorter than $AB$.
\item If $B$ is a cycle point, then $\hat F(A_1B_1)$ is shorter than $AB$. 
\end{enumerate}
\end{definition}

The next lemma explains how to attach ``very long'' critical paths, while preserving the cycle decomposing nature of the modification. 

\begin{lemma}\label{lem: technical-1}
Let $\widehat{F}: \widehat{\Gamma} \to \Delta$ be a cycle decomposing $H$-modification of $F$. Attach a very long critical path $w$. The new graph admits a cycle decomposing $H$-modification. Further, the image of any edge in a minimal critical tree is unchanged as compared with $\hat F$, under the natural identification of a metric graph as a subspace of a modification. 
\end{lemma}

\begin{proof}
Let $A$ and $B$ be the endpoints of the new critical path, with $A$ being the critical point. We break the proof into two cases based on whether $B$ is located on an existing critical path (Case I) or on the cycle (Case II).

\noindent
\underline{\sc Case I.}
Suppose $w$ is attached to a critical path connecting a critical point $A_2$ to a vertex $B_2$ on the cycle. By the trivalency assumption, $B$ and $B_2$ are distinct. Assume that $|AB| > |\widehat{F}(A_2B)|$. Then there exists a unique point $C$ in the interior of $AB$ such that $|AC| = |\widehat{F}(A_2B)|$. We proceed with the following modification.

Let $T \subset \Delta$ be the tree attached to $\mathbb{R}$ at $0$ as part of the modification. We construct a new map on $\widehat{\Gamma} \cup w$ as follows. Glue $T$ to $\widehat{\Gamma} \cup w$ by identifying $\widehat{F}(A_2B)$ with $AC$. Define the map
\[
\widehat{\Gamma} \cup w \to \Delta
\]
by mapping $\widehat{\Gamma}$ via $\widehat{F}$, mapping $T$ identically with slope $1$, and contracting the segment $BC$ to the single point $\widehat{F}(B)$. This map is harmonic.

We complete the $H$-modification by applying Lemma~\ref{lemma:glueing}, together with Proposition~\ref{theorem:trees}, to handle the contracted segment $BC$. By construction, the length of the image of $AB$ is at least the length of $A_2 B_2$. Consequently, the resulting $H$-modification is cycle-decomposing. This concludes the proof for \textsc{Case I}.

\noindent
\underline{\sc Case II.} Suppose $w$ is attached to a cycle vertex $B$ and does not meet a critical path. By assumption, $\hat F(A_1B_1)$ is shorter than $AB$. Thus, there exists a point $C$ in the interior of $AB$ such that $AC$ has the same length as $\hat F(A_1B_1)$.

Since the modification is cycle-decomposing, we can choose a radius $r$ such that, if $T_r$ is the union of $\hat F(A_1B_1)$ with the radius-$r$ ball around $\hat F(B_1)$, then $\Gamma \setminus \hat F^{-1}(T_r)$ is a disjoint union of trees $S_1, \dots, S_m$. The endpoints of these trees have the same images under $\hat F$. After possibly relabeling, assume that $w$ is attached to the tree $S_1$, whose endpoints map to $V_1 \in \Delta$.

With this radius $r$ fixed and $S_1$ identified, we proceed as follows. Let $D$ be the point on $BC$ at distance $r$ from $C$, and let $T \subset \Delta$ be the tree attached to $\mathbb{R}$ at $0$ from the modification. We construct an intermediate map $\Gamma_1 \to \Delta$ as follows:

\begin{enumerate}
    \item \textbf{Source.} Construct the domain $\Gamma_1$ starting with $\widehat{\Gamma} \cup w$. The closure $\overline{T_r}$ consists of $\hat F(A_1B_1)$ together with edges $e_1,\dots,e_l$ of length $r$ attached to $\hat F(B_1)$. Assume $V_1 \in e_1$. Glue $T$ to $w$ by identifying $\hat F(A_1B_1)$ with $AC$ and $e_1$ with $CD$.
    
    \item \textbf{The map.} Define a map $\Gamma_1 \to \Delta$ that agrees with $\hat F$ on $\overline{\Gamma \setminus S_1}$ (including all modifications), maps the attached copy of $T$ identically with slope $1$, and contracts the entire $S_1$ (with its modifications) to the single point $V_1$. This map is well-defined and harmonic.
\end{enumerate}

Finally, applying Lemma~\ref{lemma:glueing} and Proposition~\ref{theorem:trees} to the contracted region produces the desired $H$-modification. By construction, the image of $AB$ is no shorter than $A_1B_1$, and the choice of $r$ ensures that the modification is cycle-decomposing. This completes the proof of \textsc{Case II}.
\end{proof}

Given a balanced function $F\colon \Gamma\to \RR$ that contracts the cycle, if we delete the cycle, there is a forest of contracted edges that connects the cycle to the critical points. We refer to the connected components as {\it critical trees}. It is possible to have a well-spaced map $F$ with a unique critical tree. There are two essential ways in which this can happen: first, when there are two critical points $A_1$ and $A_2$ whose associated critical paths share an edge. We call this a ``$Y$-shape'', as the critical points are the top endpoints of the $Y$ and the bottom of the $Y$ attaches to the cycle. The other case we isolate is the ``$3$-flag'' case -- there is a single critical point, but there are three flags based at this point where the map is non-constant. 

The next lemma enables us to take the $Y$-shaped picture and connect it to a tree via a contracted edge that attaches to the this unique critical tree. 

\begin{lemma}\label{lem:Y_shape}
Let $F\colon \Gamma\to\RR$ be balanced and well-spaced with a unique $Y$-shaped critical tree. Let $A_1$ and $A_2$ be the critical points on the $Y$, let $C$ be its trivalent center, and let $B_1$ be the cycle point. There exists an $H$-modification $\hat F$ that is cycle decomposing, and satisfies the following conditions: if $F$ has at least $2$ critical points, there exists a critical point $A_3$ and a point $D$ in $CB_1$, such that
\begin{enumerate}[(i)]
\item there is no turning point in the interior of $A_3D$, and hence $\hat F(A_1D)$, $\hat F(A_3D)$, and $A_3D$ all have the same length, 
\item the map $\hat F$ has slope $2$ on each edge of $DB_1$, and 
\item for any critical point $A_i \notin \{A_1,A_2,A_3\}$ connected by an $F$-contracted path to $E_i$ in the interior of $CB_1$, the length of $A_iE_i$ is larger than $\hat F(A_1E_i)$. 
\end{enumerate}
\end{lemma}

The analogous lemma holds, with essentially the same proof, for the $3$-flag case, rather than the $Y$-shaped case. 

\begin{proof}
We induct on the number of critical paths. The base case is the $Y$-shaped case we have already treated. Suppose that $F\colon\Gamma\to\RR$ admits an $H$-modification satisfying the above requirements. We attach a new critical path $w$ starting at a critical point $A$ and terminating at $B$. Note that $B$ can either be a cycle point or lie on an existing critical path. In any event, we may assume that new path is longer than all other critical paths. 

Observe that if $B$ does not lie on a minimal critical path, the inequalities between edge lengths force $AB$ to be \emph{very long} in the sense described above. Similarly, if $B$ lies in the interior of $A_1D$ or $A_2D$, the inequalities imply that $AB$ is very long. In these cases, as well as when $B$ lies in the interior of $DB_1$ but $AB$ is very long, we use the construction of Case I in the previous lemma to conclude.

The only remaining case is when $B$ lies in the interior of $DB_1$ but $AB$ is not very long. Then we have
\[
|A_3D| + |DB| < |AB| \le |A_3D| + 2|DB|,
\]
so there exists a point $E$ in the interior of $DB$ such that
\[
|AB| = |A_3D| + 2|DE| + |EB|.
\]

Let $\widehat F: \widehat \Gamma \to \Delta$, and define $T = \overline{\Delta \setminus \mathbb{R}}$. Let $S \subset \text{Im}(\widehat F)$ be the union of all rays attached to $\widehat F(E)$ that do not meet $\mathbb{R}$. Mark a point $F$ in the interior of $AB$ such that $|AF| = |A_3D| + 2|DE|$.

We define an intermediate map $\Gamma_1 \to \Delta$ as follows:

\begin{enumerate}
    \item \textbf{Source.} Start with $\widehat{\Gamma} \cup AB$, removing all critical paths or modifications attached to $EB_1$. Glue a copy of $T$ by identifying $\widehat F(A_3E)$ with $AF$, and glue a copy of $S$ to $E$ by identifying $\widehat F(E)$ with $E$. Denote the resulting graph by $\Gamma_1$.
    
    \item \textbf{The map.} Map $\widehat{\Gamma}$ with $EB_1$ (and all removed critical paths or modifications) via $\widehat F$. Map the copies of $T$ and $S$ identically with slope $1$. Contract $EB_1$, $FB$, and the cycle to the single point $\widehat F(E)$. This intermediate map is well-defined and harmonic.
\end{enumerate}

Applying Lemma~\ref{lemma:glueing}, and noting that $EB_1 \cup FB \cup$ the cycle forms the $Y$-shaped configuration in Proposition~\ref{construction: genus 1 with one Y-leg_type A}, produces a cycle-decomposing $H$-modification satisfying the claim. Finally, we reattach the previously deleted critical paths in order of increasing length; each satisfies the very long condition, so the argument above applies, completing the construction.
\end{proof}

We come to the final statement. 

\begin{theorem}
Let $F\colon\Gamma\to\RR$ be a balanced function on a tropical curve of genus $1$. Assume it contracts the cycle and satisfies the genericity conditions of Lemma~\ref{lem: generic-lengths}. Then $F$ admits a cycle decomposing $H$-modification. 
\end{theorem}

\begin{proof}
First, suppose there are two different critical trees each containing a critical path of minimal length. Then we apply the construction made in Proposition~\ref{prop:g1_speyer_two_legs}.
Viewing the rest of the critical paths as being attached successively in increasing order of length, they are all very long, and we're done by the previous lemma. 

Otherwise, there is a unique critical tree containing a critical path of minimal length. It can either be $Y$-shaped or the $3$-flag case; we'll treat the former and leave the latter to the reader. 

By Lemma~\ref{lem:Y_shape}, it suffices to consider attachments of new critical paths that are not attached to the $Y$-shaped tree. We proceed by induction on the number of such attachments, and attach them in increasing order of length. 

Otherwise, there is a unique critical tree $T$ containing a critical path of minimal length. 
This tree can be built upon either the $Y$-shaped configuration or the $3$-flag configuration; we treat the former explicitly here, leaving the latter to the reader.

We begin with the ``bare'' $Y$-shaped configuration $T$, removing all other critical path attachments and ordering them by increasing length. We then reattach these paths in this order. Suppose that the first $k-1$ paths (with $k \ge 1$) are attached to $T$, while the $k$-th path is the first one not attached. Consider the union of $T$ with the first $k-1$ attachments. Applying Proposition~\ref{lem:Y_shape} produces a cycle-decomposing $H$-modification with a distinguished critical vertex $A_3$ and a point $D$ on $CB_1$ having the desired properties.

If the $k$-th attachment does not exist (i.e., all paths attach to $T$), the proof is complete. Otherwise, let $AB$ denote the $k$-th attachment. If $AB$ is very long, we proceed exactly as in Case II of Lemma~\ref{lem: technical-1}.

Now assume $AB$ is not very long. Temporarily remove all critical paths attached to the interior of $DB_1$. Next, apply the slope adjustment technique (Remark~\ref{rem: slope-adjustment}): select a point $E \in DB_1$ and reduce the slope along the segment $DE$ from $2$ to $1$, choosing $E$ so that the slope-weighted length of $A_3B_1$ equals that of $AB$. With this adjustment, construct a cycle-decomposing $H$-modification using an argument analogous to Proposition~\ref{lem:Y_shape}, substituting the $Y$-shaped configuration in the gluing step with the construction from Proposition~\ref{prop:g1_speyer_two_legs}.

Finally, collect all previously removed paths—those from the interior of $DB_1$ and those removed at the beginning of the proof—and reattach them in order of increasing length. In this new configuration, every attachment satisfies the very long condition. The proof then concludes by applying Case I or Case II of Lemma~\ref{lem: technical-1} as appropriate.
\end{proof}


\section{Genus two results}\label{sec: genus 2}

In this section, we prove Theorems~\ref{thm: genus-2-A} and~\ref{thm: genus-2-B}, along with related results. We begin by establishing the existence of $H$-modifications in several base cases via explicit constructions. We then use moduli-theoretic considerations to explain why ``sufficiently long'' trees may be grafted onto these constructions, thereby obtaining the general results. 

\subsection{Explicit constructions}\label{sec: genus2-explicit-constructions}

Any abstract tropical curve of genus $2$ can be stabilized by pruning trees and removing bivalent vertices. In the trivalent case, the resulting graph is necessarily either a $\Theta$-graph or a dumbbell graph. The complexity in the two cases is analogous, so we focus primarily on the $\Theta$ case and include a representative result for the dumbbell case at the end of the section. 

We study balanced functions that contract a genus $2$ subgraph. A dimension count suggests that realizability should impose two conditions on the moduli of the domain curve. We exhibit several distinct ways in which these two conditions may arise. 

We begin with the genus $2$ analogue of Proposition~\ref{prop:g1_speyer_two_legs}. Here, the codimension-$2$ condition on moduli appears as a direct generalization of Speyer's condition. 

\begin{proposition}\label{prop: genus2-basic}
Consider the map $F\colon\Gamma\to\mathbb{R}$ described in Figure~\ref{fig: genus2-figure} below. 
\begin{figure}[h!]
\includegraphics{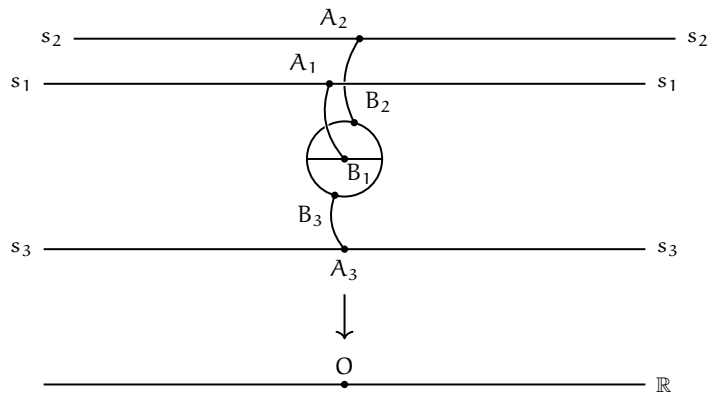}
\caption{The map depicted in this figure contracts the $\Theta$ subgraph, as well as the edges $A_iB_i$. It has slope $s_i$ along the labeled rays, in the indicated direction.}\label{fig: genus2-figure}
\end{figure}
Suppose the three edges labeled $A_iB_i$, attached to distinct edges of the $\Theta$, all have the same length $\ell$, and that these edges together with the $\Theta$ subgraph are contracted to $O$. Then $F$ admits an $H$-modification and is therefore realizable.
\end{proposition}

\begin{proof}
We make the following modifications: add a ray to $\mathbb{R}$ at $O$; mark the point $P$ on the ray which has distance $\ell$ from $O$, and add another ray at $P$. Each $A_i B_i$ is mapped to the edge with slope $1$, and the two path components of $\Theta\setminus\{B_1,B_2,B_3\}$ are mapped onto the two rays extending from $P$ respectively. This is depicted in Figure~\ref{fig: initial-genus2basic} below. 
\begin{figure}[h!]
\includegraphics[scale=0.75]{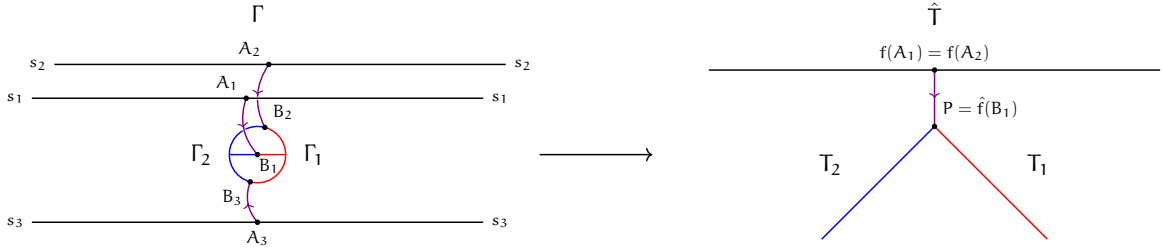}
\caption{The initial step in the modification required for Proposition~\ref{prop: genus2-basic}.}\label{fig: initial-genus2basic}
\end{figure}
Now consider one of the path components of $\Theta\setminus\{B_1,B_2,B_3\}$ and the ray in $T$ which it is mapped to. Let $Q$ be the trivalent point in the path component and choose the edge which has shortest length $l'$. At the point with distance $l'$ away from $P$ on the added ray in $T$, add another ray. We may assume the edge is $B_1 Q$. Mark the points $Q_2$ and $Q_3$ on the other edges which have distance $l'$ from $B_2$ and $B_3$ respectively, and add a ray at $Q_2$ and $Q_3$. At the midpoints of $Q Q_2$ and $Q Q_3$, add two rays. Map the edges and rays to $\hat{\Gamma}$ all with slope $1$. The local picture is depicted in Figure~\ref{fig: second-step-genus2-basic}.
\begin{figure}[h!]
\includegraphics[scale=0.75]{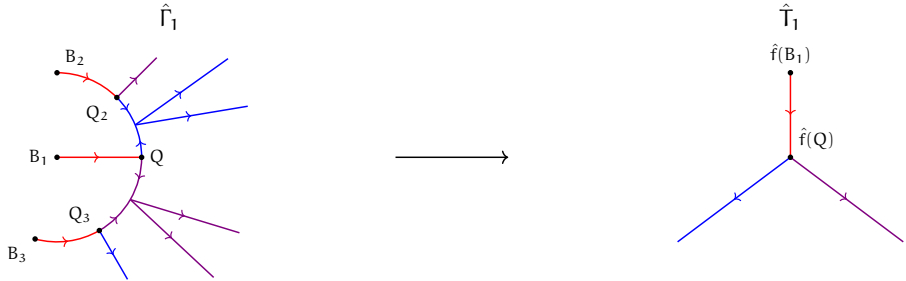}
\caption{The modifications required on the $\Theta$ subgraph.}\label{fig: second-step-genus2-basic}
\end{figure}
Repeat the same set of modifications for the other path component, and balance the slopes at $A_1,A_2,A_3$ the same way as in Proposition \ref{prop:g1_speyer_two_legs}. Then we obtain an $H$-modification of $F\colon\Gamma\to\RR$, and $F$ is realizable.
\end{proof}

Based on this construction, one might guess that ``three edge lengths equal'' is a perfect generalization to the genus $2$ case. But the next proposition shows that the picture is more subtle. 

A tropical curve of genus $2$ is equipped with a unique involution whose quotient is a tree -- the ``tropical hyperelliptic involution''. The fixed points of this involution are {\it Weierstrass points} and points that are swapped by the involution are called {\it conjugate}. For a $\Theta$ graph, the Weierstrass points are the midpoints of each of the three edges, and the conjugate points are pairs of points whose center of mass is the midpoint. 

\begin{proposition}\label{construction: theta with two symmetric T-legs}
Consider the map $F\colon\Gamma\to\mathbb{R}$ depicted in Figure~\ref{fig: genus2-figure-2}.
\begin{figure}[h!]
\includegraphics{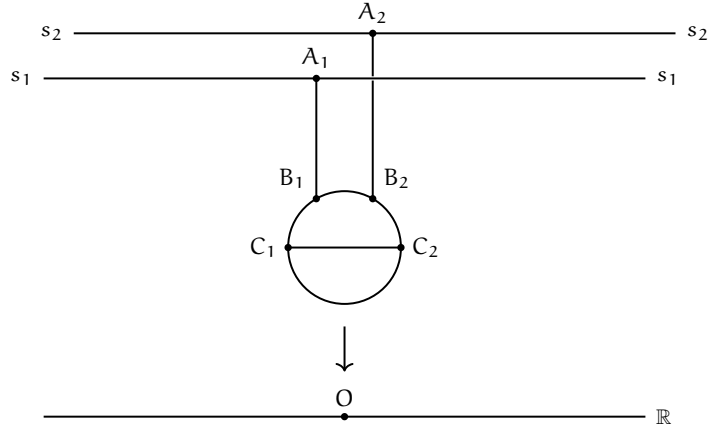}
\caption{The map in this figure again contracts the $\Theta$ subgraph, as well as the edges $A_iB_i$. It has slope $s_i$ along the labeled rays, in the indicated direction.}\label{fig: genus2-figure-2}
\end{figure}
Suppose the edges $A_1B_1$ and $A_2B_2$ have equal length, and further that $B_1$ and $B_2$ are conjugate points, i.e. the lengths $B_1 C_1=B_2 C_2$. There exists an $H$-modification of $F$, and so it is realizable.
\end{proposition}

\begin{proof}
Add a ray at $O$ in the target $\RR$, and rays on the added ray at distances $A_1B_1$ and $A_1B_1+B_1C_1$ from $O$ respectively. Do the same modification $s_1-1$ and $s_2-1$ times at $A_1,A_2\in\Gamma$. These are mapped naturally to $\hat{T}$ with slope $1$ on every edge. Also add two rays to the midpoint of $B_1 B_2$, add two rays to the midpoints of the two edges connecting $C_1$ and $C_2$. Map the edges and rays as shown in Figure~\ref{fig: genus2-weierstrass-proof}.
\begin{figure}
\includegraphics[scale=0.75]{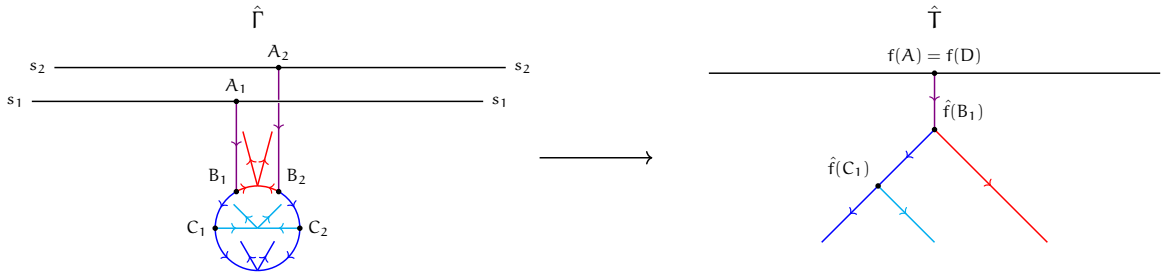}
\caption{The modifications required in the proof of Proposition~\ref{construction: theta with two symmetric T-legs}.}\label{fig: genus2-weierstrass-proof}
\end{figure}
The modification $\hat{F}$ is an $H$-modification.
\end{proof}

Realizability is a closed condition, the proposition also holds when both attaching points are a single Weierstrass point. However, at the Weierstrass points, new phenomena appear. 

\begin{proposition}\label{prop: genus-2-Weierstrass}
Consider the map $F\colon\Gamma\to\mathbb{R}$ depicted in Figure~\ref{fig: genus-2-Weierstrass}.
\begin{figure}[h!]
\includegraphics{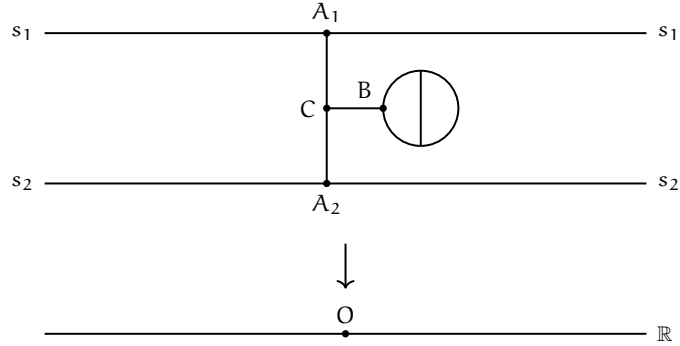}
\caption{The map $F$ contracts the $\Theta$ subgraph, as well as the edges $A_iC$ and $BC$.}\label{fig: genus-2-Weierstrass}
\end{figure}
Suppose $A_1 C$ and $A_2 C$ have equal lengths, and further that $B$ is Weierstrass, or in other words, self-conjugate. There exists an $H$-modification of $F$, and so it is realizable.
\end{proposition}
\begin{proof}
    The proposition follows by combining the construction in Proposition~\ref{construction: genus 1 with one Y-leg_type A} with Proposition~\ref{construction: theta with two symmetric T-legs} above. The details are left to the reader. 
\end{proof}

We expect that there are analogues of the theorems above in the other combinatorial type of genus $2$ curve -- the dumbbell graph. These are building blocks of the now-famous ``chain of cycles'', appearing in~\cite{CDPR} and later developments in Brill--Noether theory. We include the following representative construction for the dumbbell. 

\begin{proposition}
Consider the map $F\colon\Gamma\to\RR$ depicted in Figure~\ref{fig: genus-2-chain-of-loops} below. 
\begin{figure}[h!]
\includegraphics{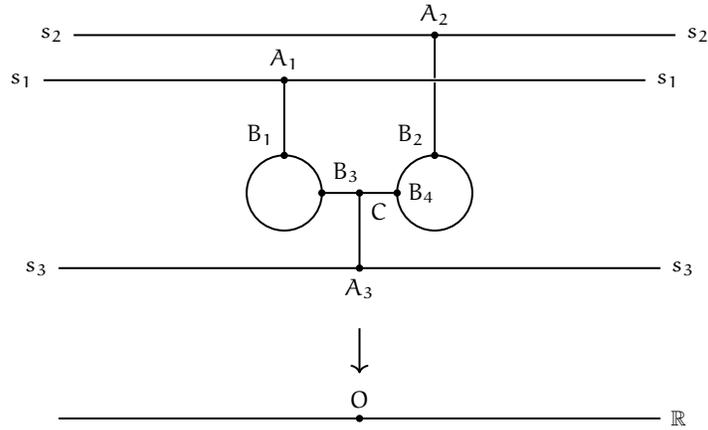}
\caption{The map $F$ contracts the genus $2$ subgraph, as well as all the edges incident to this subgraph.}\label{fig: genus-2-chain-of-loops}
\end{figure}
Suppose the distance between $A_1$ and $B_1$ is the same as the distance between $A_3$ and $B_3$, and the distance between $A_3$ and $B_4$ is the same as the distance between $A_2$ and $B_2$. There exists an $H$-modification of $F$, and so it is realizable. 
\end{proposition}

\begin{proof}
An $H$-modification can be obtained by similar constructions as ones we have already made. The modification is depicted below:
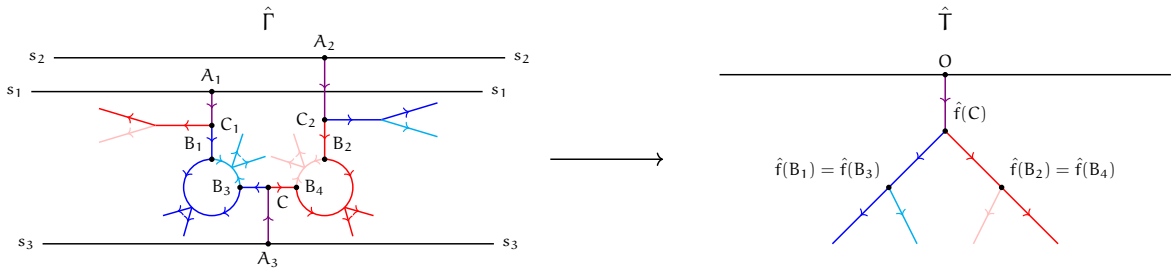
\[
\includegraphics[scale=0.75]{Figures/g2_dumbbell_shared_leg_modification.tex}
\]
Each cycle is mapped using the construction of Remark~\ref{rem: slope-adjustment} and the trees based at $C_1$ and $C_2$ are also analogous to that remark.  
\end{proof}

\subsection{Appending long trees}

Let $F\colon \Gamma \to \RR$ be a balanced function on a tropical curve. A subgraph $\Gamma_0 \subset \Gamma$ is called a \emph{frame} if $\Gamma \setminus \Gamma_0$ is a forest and each connected component of this forest is attached to the frame by an $F$-contracted edge. These contracted edges will be called \emph{connecting edges}. A frame is \emph{realizable} if it is realizable as a tropical map.

Our goal is to prove the following.

\begin{theorem}\label{thm: long-scaffold}
Let $F\colon \Gamma \to \RR$ be a balanced function with a realizable frame $\Gamma_0$. If the connecting edges are all sufficiently long, then $F$ is realizable.
\end{theorem}

In the situations of interest, the frames will be the realizable tropical maps constructed in Section~\ref{sec: genus2-explicit-constructions}. However, we present the argument in a way that allows any independently constructed realizable tropical maps to serve as frames.

Together with the constructions of Section~\ref{sec: genus2-explicit-constructions}, this theorem immediately implies Theorems~\ref{thm: genus-2-A} and~\ref{thm: genus-2-B}.

\subsubsection{Preparations: tropicalization for the double ramification locus}

We deduce the theorem from formal properties of tropical compactifications~\cite{Tev07,U15} and extended tropicalizations~\cite{ACP}, applied to the double ramification locus in $\mathcal M_{g,n}$.

Fix a vector $A = (a_i) \in \mathbb Z^n$ with vanishing sum. The double ramification locus
\[
{\sf D}_g(A) \subset \mathcal M_{g,n}
\]
parameterizes pointed curves $[C,p_1,\ldots,p_n]$ such that $\mathcal O_C\!\left(\sum a_i p_i\right)$ is trivial.

Let $K$ be a non-archimedean valued field. By results of Abramovich--Caporaso--Payne\footnote{The results there are stated in the language of Berkovich spaces. We require substantially less here and therefore avoid this more technical perspective.}, there is a natural \emph{tropicalization} map
\[
\mathsf{trop}\colon \mathcal M_{g,n}(K)\to \mathcal M_{g,n}^{\sf trop},
\]
from the $K$-points of the moduli space of smooth pointed curves to its tropicalization. This map sends a curve over $K$ to the metrized dual graph of its stable model, as described in Section~\ref{sec: tropicalizations}.

Composing with the inclusion
\[
{\sf D}_g(A)(K)\hookrightarrow \mathcal M_{g,n}(K),
\]
we obtain a map whose domain parameterizes exactly those curves that admit a map to $\mathbb P^1$ with ramification orders prescribed by $A$ at the marked points. Such a map, when it exists, is unique up to scalar.

The image of this map is precisely the locus of tropical curves that admit a balanced function with slope $a_i$ along the $i^{\rm th}$ ray. Again, such tropical maps are unique up to translation when they exist. Thus, this image is exactly the realizable locus of interest.

Since this construction works for any non-archimedean valued field, we choose $K$ whose valuation is surjective onto $\mathbb R$. The image of ${\sf D}_g(A)$ under $\mathsf{trop}$ is then called its tropicalization. The structure theory of tropicalizations implies the following~\cite{ACP,U15}.

\begin{proposition}
The tropicalization of ${\sf D}_g(A)$ is the support of a cone complex in $\mathcal M_{g,n}^{\trop}$. 
\end{proposition}

\begin{proof}
Formal from~\cite{U15}. 
\end{proof}

Denote the tropicalization by ${\sf D}_g^{\sf trop}(A)$ and choose any cone complex structure on it. By toroidal geometry, this subcomplex of $\mathcal M_{g,n}^{\sf trop}$ determines an open subset of a toroidal modification of of the moduli space of curves $\Mbar_{g,A}$. By the basic theorems of tropical compactifications~\cite{Tev07,U15}, we have the following. 

\begin{proposition}
The closure of ${\sf D}_g(A)$ in  $\mathcal M_{g,A}$ is proper. Moreover, this closure meets the toroidal strata of $\mathcal M_{g,A}$ in the expected dimension. 
\end{proposition}

We denote the closure by $\overline {\sf D}_g(A)$. 

We consider compactifications on the tropical side. Every cone complex $\Sigma$ admits a canonical compactification~\cite{ACP}. Every cone $\sigma$ can be expressed as $\Hom(P,\RR_{\geq 0})$ where $P$ is a finitely generated monoid. By replacing $\RR_{\geq 0}$ with $\RR_{\geq 0}\cup\{\infty\}$ with the obvious monoid structure, one compactifies the cone. Gluing the compactified cones together, we obtain a compactification
\[
\Sigma\hookrightarrow\overline\Sigma. 
\]
The complement of $\Sigma$ decomposes into a disjoint union of cone complexes, called the {\it infinite faces}. See~\cite[Section~2]{ACP} for further details. 

Now, just as $K$-points of ${\sf D}_g(A)$ map onto the cone complex ${\sf D}_g^{\sf trop}(A)$, the $K$-points of the compactification space $\overline{\sf D}_g(A)$ map onto a compactification $\overline{\sf D}_g^{\sf trop}(A)$, so we have:
\[
\begin{tikzcd}
 {\sf D}_g(A)(K)\arrow[hookrightarrow]{r}\arrow{d}& \overline {\sf D}_g(A)(K)\arrow{d}\\
 {\sf D}_g^{\sf trop}(A)\arrow[hookrightarrow]{r}& \overline{\sf D}_g^{\sf trop}(A).
\end{tikzcd}
\]

We will use this diagram to complete the proof of our genus $2$ results as follows. First, we will argue that certain points at infinity lie in the image of the right vertical map. We will then deduce the existence of points in the interior formally. 

\begin{remark}
One can also deduce the results by using properties of tropicalization over rank $2$ valued fields, but this relies on results that we are unable to locate appropriate references for. 
\end{remark}

\subsubsection{Realizability in the infinite faces}
We now prove Theorem~\ref{thm: long-scaffold}.

\begin{proof}
Let $F\colon \Gamma\to\RR$ be a balanced function and let $\Gamma_0$ be an allowable frame. Let $e_1,\ldots,e_m$ be the contracted edges connecting the frame to the forest, and let $T_1,\ldots,T_m$ be the trees comprising this forest. Viewing the half-edge of $e_i$ that is incident to $T_i$, we obtain a balanced map $T_i\to\RR$ in which this ray is contracted.

Consider the restriction $\Gamma_0\to \RR$ and shrink all bounded edges of $\Gamma_0$ to length $0$. The result is a graph $\Gamma_0^\star$ consisting of a unique vertex, which we label with genus equal to the genus of $\Gamma_0$. Attach all the contracted connecting edges to this vertex. The resulting tropical map is realizable, for example by closedness of the realizable locus.

By varying the lengths of $e_1,\ldots,e_m$, we obtain a cone of realizable tropical curves and thereby a cone in $\overline{\sf D}_g^{\sf trop}(A)$. Since the latter is closed, the point on the face at infinity corresponding to setting all these edge lengths equal to infinity is also realizable. Consider the map\footnote{Strictly speaking, we should allow the target to be compactified as well. However, since the infinite edges are all contracted, this is unnecessary.} 
\[
\Gamma^\infty \to \RR
\]
obtained by setting the lengths of the $e_i$ equal to infinity.

We claim that $\Gamma^\infty\to\RR$ is realizable, in the sense that it lies in the image of the tropicalization map above. Recall from~\cite{ACP,R16} that infinite edges correspond algebraically to tropicalizations of families of maps over valued fields with nodal domains: the nodes of the domain curve appearing in the generic fibers correspond to the infinite edges.

To construct a realization of $\Gamma^\infty\to\RR$, begin with a realization $C_0\to\PP^1$ of $\Gamma_0\to\RR$. Choose realizations $C_i\to\PP^1$ of the trees $T_i\to\RR$ connecting to $\Gamma_0$, with distinguished marked points $q'_i$ corresponding to the edges $e_i$. For each $i$, choose a point $q_i$ on $C_0$ such that under the tropicalization map
\[
C_0(K)\to\Gamma_0,
\]
the point $q_i$ maps to the attaching point of $e_i$ on the frame $\Gamma_0$.

Form a nodal curve $C$ by identifying $q_i$ with $q'_i$ for each $i$. After possibly scaling the maps $C_i\to\PP^1$, the points $q_i$ and $q'_i$ map to the same point of $\PP^1$. This yields an induced map
\[
C\to\PP^1,
\]
which by construction realizes $\Gamma^\infty\to\RR$.

We have therefore shown that the point of $\overline{\sf D}_g^{\sf trop}(A)$ corresponding to $\Gamma^\infty\to\RR$ is realized. Since $\overline{\sf D}_g^{\sf trop}(A)$ is a compactified cone complex, there must exist a cone in ${\sf D}_g^{\sf trop}(A)$ whose closure contains this face at infinity. The theorem follows.
\end{proof}

\section{Generalizations via dimensional reduction}\label{sec: bootstrap}

We have thus far examined maps $f\colon \Gamma\to\RR$. However, one is often interested in maps to higher-dimensional vector spaces. There is a simple bootstrapping technique that allows one to deduce statements about maps to $\mathbb R^r$ from statements about balanced functions. 

\subsection{Setup}
We recall some basic terminology. Let $\Gamma$ be a tropical curve. A {\it tropical map} $\Gamma\to\RR^r$ is a map such that for any linear projection $\RR^r\to\RR$, the induced map
\[
\Gamma\to\RR^r\to\RR
\]
is balanced. It is sufficient to check this for the coordinate projections with respect to any coordinate system. 

Recall that a balanced map $f\colon \Gamma\to \RR^r$ has a {\it type}. Loosely, this is the data obtained by discarding the metric data but retaining all other discrete data associated to $f$. Formally, it consists of the following:
\begin{enumerate}[(i)]
\item The marked graph $G$ underlying $\Gamma$,
\item For each directed edge $e$ of $G$, the {\it edge direction} in $\RR^r$, or equivalently, for every linear projection $\RR^r\to \RR$, the slope of the composite along $e$,
\item For each marked ray $\ell$ of $G$, its {\it edge direction}, or equivalently, for every linear projection $\RR^r\to \RR$, the slope of the composite along $\ell$, directed away from the root of the ray. 
\end{enumerate}

If we fix the type $\Theta$, there is a space parameterizing tropical curves with a balanced map to $\RR^r$ having this type. It is naturally isomorphic to $\sigma^\circ\times \RR^r$, where $\sigma$ is a rational polyhedral cone and $\sigma^\circ$ is its interior. The cone $\sigma$ should be viewed as a conical subset
\[
\sigma\subset\RR_{\geq 0}^E,
\]
where $E$ is the edge set of $G$, consisting of those assignments of edge lengths that support a continuous piecewise linear map of type $\Theta$. In the interior, there is the additional condition that the edge lengths are positive. 

It is slightly more convenient to work with the space of maps {\it up to translation}. Given a tropical curve $\Gamma$ with underlying graph $G$ and a compatible type $\Theta$, the graph $\Gamma$ may or may not admit a map to $\RR^r$ of type $\Theta$, and if it does, there is an $\mathbb R^r$-torsor worth of such maps. 

It follows from the main result of~\cite{R16} that there is a closed conical subset $\mathsf M^{\sf real}_\Theta$ parameterizing those maps (again, up to translation) that are realizable. That is, they arise from tropicalizing a map
\[
\mathcal C\dashrightarrow \mathbb G_m^r
\]
that is regular outside a marked set of points $p_1,\ldots,p_n$ and defined over a non-archimedean field $K$. The tropicalization procedure is identical to the one presented in Section~\ref{sec: trop-rsm}, applied coordinatewise.

\subsection{The dimensional reduction} We now assume that the marked graph $G$ underlying $\Theta$ is trivalent with genus $0$ vertices -- the {\it maximally degenerate} case. Fix a linear projection $\chi\colon \mathbb R^r\to \mathbb R$. There is an associated type $(\Theta,\chi)$, which is the type of a map to $\RR$, obtained from a map of type $\Theta$ by projecting. We therefore obtain a natural realizable subset
\[
\mathsf M^{\sf real}_{\Theta,\chi}\subset \RR_{\geq 0}^E
\]
comprising those prescriptions of edge lengths such that the resulting map is realizable. 

Fix a splitting of the torus $\mathbb G_m^r$ corresponding to basis $\chi_1,\ldots,\chi_r$ of the character lattice. Since we can compose a realization of $[F\colon \Gamma\to\RR^r]$ with the character $\chi_i$ to produce a realization of $[\chi_i\circ F]$ there is a containment:
\[
\mathsf M^{\sf real}_{\Theta} \subset \bigcap_{i=1}^r \mathsf M^{\sf real}_{\Theta,\chi_i}.
\]
The following statement allows us to conclude the reverse containment under some circumstances. The result is, in some sense, implicit in~\cite{JR17}. 

Recall that $g$ denotes the genus of the marked graph $G$, which by assumption is equal to the first Betti number of $G$.

\begin{theorem}
Let $p$ be a point of $\bigcap_{i=1}^r \mathsf M^{\sf real}_{\Theta,\chi_i}$ and assume that $\bigcap_{i=1}^r \mathsf M^{\sf real}_{\Theta,\chi_i}$ has codimension $rg$ in a neighborhood of $p$. Then $p$ lies in  $\mathsf M^{\sf real}_{\Theta}$.
\end{theorem}

\begin{proof}
We have fixed $\Theta$. Let $A = (c_{ij})$ be the $r\times n$ matrix recording the slopes $c_{ij}$ of the characters $\chi_i$ along the marked rays $p_j$. As discussed in the previous section, there is an associated locus
\[
\mathsf{D}_g(A)\subset\mathcal M_{g,n}
\]
in the moduli space of smooth pointed curves comprising those curves that admit a rational map to $\mathbb G_m^r$, regular away from the marked points, such that the vanishing or pole order of the character $\chi_i$ along the marked point $p_j$ is exactly $c_{ij}$. Note that if a curve admits such a map, it admits a unique one up to $\mathbb G_m^r$-translation. 

Now, if $K$ is a valued field, there are natural maps on $K$-points:
\[
\mathsf{D}_g(A)(K)\subset\mathcal M_{g,n}(K)\xrightarrow{\sf trop} \mathcal M_{g,n}^{\sf trop},
\]
sending a map $[f\colon \mathcal C\dashrightarrow\mathbb G_m^r]$ to the domain curve of its tropicalization.\footnote{Since both algebraic and tropical maps of this form are unique up to the appropriate translation when they exist, there is no loss of information in remembering only the underlying curve.}

There is also a moduli map $\mathbb R_{\geq0}^E\to \mathcal M_{g,n}^{\sf trop}$. The locus $\mathsf M^{\sf real}_{\Theta}$ is the preimage under this map of the tropicalization ${\sf trop}(\mathsf{D}_g(A)(K))$. As usual, we compress the notation to ${\sf trop}(\mathsf{D}_g(A))$, since the precise choice of field $K$ will not play a significant role. 

Let $A_i$ be the matrix obtained from $A$ via the projection $\chi_i\colon\RR^r\to\RR$. Analogously, we may identify $\mathsf M^{\sf real}_{\Theta,\chi_i}$ with the tropicalization of $\mathsf{D}_g(A_i)$. It is certainly true that
\[
\mathsf{D}_g(A) = \bigcap_{i=1}^r \mathsf{D}_g(A_i).
\]
Thus the proposition becomes a statement about when tropicalization commutes with intersection. By general results on tropical intersections, these two operations commute whenever the intersection is locally of the expected dimension~\cite{OP}; see also~\cite{He19,OR,R12}. 

The statement of the theorem then follows from a Riemann--Hurwitz computation: the spaces $\mathsf{D}_g(A_i)$ have codimension $g$ in $\mathcal M_{g,n}$.
\end{proof}

For example, by applying the theorem above, the $r = 1$ case of the realizability of Speyer's theorem, proved in Section~\ref{sec: genus 1}, implies the general case of well-spacedness for trivalent domains. Similarly, the genus $2$ results for $r=1$ imply results in higher dimensional targets. However, one should note that there exist more exotic forms of non-realizability for tropical curves of higher genus. See recent work of Koyama~\cite{Koy23}. 

It should be possible to modify the theorem to work without the maximal degeneracy hypothesis. However, our goal is to show the utility of the methods, we do not go into this here.

\bibliographystyle{siam} 
\bibliography{HurwitzRealizability} 

\end{document}